\documentclass[11pt]{article}
\usepackage{amsmath,amsfonts,amsthm,amssymb,blkarray}
\usepackage{mathrsfs, graphicx,color,supertabular, booktabs, graphicx}
\usepackage{ulem}
\usepackage{soul}
\usepackage{color, xcolor}
\usepackage{tikz}
\usetikzlibrary{arrows.meta}
\usetikzlibrary{patterns,arrows,decorations.pathreplacing}
\newtheorem{thm}{Theorem}[section]
\newtheorem{lem}[thm]{Lemma}
\newtheorem{cor}[thm]{Corollary}

\newtheorem{Def}[thm]{Definition}
\newtheorem{prop}[thm]{Proposition}

\newcommand\scalemath[2]{\scalebox{#1}{\mbox{\ensuremath{\displaystyle #2}}}}

\title{\bf \large On the divisibility of H-shape trees and their spectral determination}

\author{
{\small Zhen Chen$^{a}$, \ \  Jianfeng Wang$^{a,}$\footnote{Corresponding author.
\newline{\it \hspace*{5mm}Email addresses:} chenzzhen@126.com(Z. Chen), jfwang@sdut.edu.cn (J.F. Wang), mbrunett@unina.it, fbelardo@gmail.com(F.Belardo)}\;, \ \ Maurizio Brunetti$^b$, \ \ Francesco Belardo$^b$}\\[2mm]
\footnotesize $^a$School of Mathematics and Statistics, Shandong University of Technology, Zibo 255049, China\\
\footnotesize $^b$Department of Mathematics and Applications, University of Naples Federico II, Naples, Italy}
\date{ }
\begin{document}

\maketitle

\begin{abstract}

A graph $G$ is divisible by a graph $H$ if the characteristic polynomial of $G$ is divisible by that of $H$.  In this paper, a necessary and sufficient condition for recursive graphs to be divisible by a path is used to show that the H-shape graph  $P_{2,2;n-4}^{2,n-7}$, known to be (for $n$ large enough) the minimizer of the spectral radius among the graphs of order $n$ and diameter $n-5$, is determined by its adjacency spectrum if and only if $n \neq 10,13,15$. \\



\end{abstract}

\baselineskip=0.202in

\section{\large Introduction}\label{S1}
In this paper, we only deal with undirected simple graphs.
Let $G = (V_G,E_G)$ be a graph with order $\nu_G$, size $\varepsilon_G$ and adjacency matrix $A(G)$. The {\it spectrum} ${\rm sp}(G)$ of $G$ is the multiset of eigenvalues of $A(G)$, i.e.\ the roots of  the the characteristic polynomial of $G$ $\phi(G) = \phi(G,\lambda) :=\det(\lambda I - A(G))$. The matrix $A(G)$ is symmetric; thus, its eigenvalues are real and we denote them by $\lambda_1(G) \geqslant \lambda_2(G) \geqslant  \cdots \geqslant \lambda_{\nu_G}(G)$. The maximum eigenvalue $\lambda_1(G)$ of $G$ is called the {\it index} of $G$ and it is often denoted by $\rho(G)$. From the Perron-Froebenius theorem it follows that $\rho(G)$ has multiplicity $1$ and it is also equal to the spectral radius of $G$, i.e.\ the number $\max \{ \lvert \lambda_i (G) \rvert \mid 1 \leqslant i \leqslant \nu_G \}$.

Two non-isomorphic graphs $G$ and $H$ are said to be {\it ($A$-)cospectral mates} if they share the same spectrum. Clearly, this happens if and only if $\phi(G) = \phi(H)$. A graph $G$ is said to be {\it determined by its spectrum} (or DS for short) when $G$ has no cospectral mates or, equivalently, when ${\rm sp}(G) = {\rm sp}(H)$ only if $G$ and $H$ are isomorphic.

An important research area of spectral graph theory is devoted to establish which graphs are
determined by their spectrum. Along the years, this problem has attracted the interest of many researchers \cite{dam,dam1}.

Despite the large amount of achievements already available in literature, there is still much work to do: for instance,
a celebrated Schwenk's theorem states that almost trees are cospectral  \cite{Sch-tree}, but the quest launched by  van Dam and Haemers \cite{dam} for the identification of the trees which are DS is still on. The problem has been attacked for  the T-shaped tree \cite{wang1}, the starlike trees \cite{gha1,lep-gut}, the daggers \cite{lhl} and the trees with spectral radius at most $\sqrt{2+\sqrt{5}}$ \cite{gha-omi-tay,she-hou-zhang}. Few partial results have been obtained for
the H-shape graphs \cite{hu, liu_huang} (see the next paragraph for the definitions).

Let $P_n$ and $C_n$ respectively denote the path and the cycle with $n$ vertices. We label their vertices by $0,1, \dots, n-1$ assuming that consecutive integers correspond to adjacent vertices.
For $0 < m_1 < \cdots < m_t < r-1$, we denote by
$P_{n_1,n_2,\dots,n_t;r}^{m_1, m_2, \dots,m_t}$ the graph obtained from $P_r$ by attaching at its vertex $m_i$ a pendant path of $n_i$ edges  for each $i = 1,2,\dots,t$. Similarly, for $0 \leqslant m_1 < \cdots < m_t \leqslant r-1$,
we denote by $C_{n_1,n_2,\dots,n_t;r}^{m_1,m_2,\dots,m_t}$ the graph obtained from $C_r$ by attaching at its vertex $m_i$ a pendant path of $n_i$ edges  for each $i = 1,2,\dots,t$ (see Fig.~\ref{Fig2}).
After \cite{woo}, the graphs of type $P_{n_1,n_2,\dots,n_t;r}^{m_1, m_2, \dots,m_t}$  and  $C_{n_1,n_2,\dots,n_t;r}^{m_1,m_2,\dots,m_t}$ are respectively known as {\it open} and {\it closed quipus}, and they can be structurally characterized as the trees (resp.,\ unicyclic graphs) with maximum vertex degree $3$ such that the vertices of degree $3$ all lie on a path (resp.,\ a cycle). Open quipus with $t=1$ are also known as T-shape graphs; whereas open quipus with $t=2$ are sometimes called H-shape or $\Pi$-shape trees (see Fig.~\ref{Fig1}). Closed quipus with just one pendant path are called lollipops or tadpole graphs. \\

In this paper we focus on the spectral determination of the graphs in the family
\begin{equation}\label{accan}  \mathscr{H}:= \{ H_n\}_{n\geqslant 10}, \qquad \text{where $H_n =  P_{2,2,n-4}^{2,n-7}$}.
\end{equation}
Our works is part of a larger project concerning the spectral determinations of the graphs $G_{n,D}$'s minimizing the spectral radius in the set of graphs with $n$ vertices and diameter $D$. As  matter of fact, van Dam and Kooij
\cite{dam2} conjectured that the open quipu \[ OQ_{n,e}= P_{\lfloor \frac{e-1}{2} \rfloor, \lceil \frac{e-1}{2} \rceil; n-e+1}^{\lfloor \frac{e-1}{2} \rfloor, n-e- \lceil \frac{e-1}{2} \rceil}\] is one of those minimizer for $D=n-e$ and $n$ large enough, identifying  $G_{n,D}$ when $D \in \{1,2,\lfloor \frac{n}{2} \rfloor, n-3,n-2,n-1\}$.

After \cite{cio-dam-koo-lee,dam2,yuan-shao-liu}, we know that the van-Dam-Kooji conjecture holds
for $1\leq e \leq 5$, whereas it fails for $e\geqslant 6$ (see \cite{cio-dam-koo-lee,lan-lu,lan-lu-shi,lan-shi,sun}).

The spectral determination of $G_{n,D}$ with $D \in \{n-1,n-2,n-3,n-4\}$ has been performed in \cite{gao-wang, gha-omi-tay}. Next theorem is our main result and, as already announced, concerns
\[ H_n=P_{2,2;n-4}^{2,n-7}= OQ_{n,5}= G_{n,n-5}. \]
\begin{thm}\label{main}
The only graphs in the set $\mathscr{H}= \left\{ H_n \mid n \geqslant 10 \right\}$ which are not  DS are $H_{10},H_{13}$ and $H_{15}$.
\end{thm}
A graph $G$ is said to be {\it divisible} by a graph $H$ if $\phi(H)$ divides  $\phi(G)$.
In some cases divisibility can be useful to establish whether a graph $G$ is DS or not. In fact,
 $G$ is cospectral to the disjoint union of graphs  $\bigsqcup_{i=1}^sG'_s$ if and only if
$\phi (G) = \prod_{i=1}^s \phi(G'_i)$. Moreover, since $\mathbb Z[x]$ is a unique factorization domain, if it happens that, for every connected graph $G'$ dividing $G$,
$\phi(G)/\phi(G')$ is the characteristic polynomial of a graph only if $G'$ is isomorphic to $G$, then $G$ is DS (note that, conventionally, $1$ the characteristic polynomial of the null graph $P_0$, which has no vertices and, {\it a fortiori}, no edges).
This is the motivation for studying in this paper divisibilities graphs.\medskip

We now explain how the paper is organized. In Section 2 we collect the basic tools, most of which are well-known to the scolars. Section 3 is mainly devoted to divisibilities; in particular, we find conditions for sequences of recursive graphs to be divisible by a path and, for a graph in $\mathscr{H}$, to be divisible by a cycle and some specific closed quipus. In Section 4, further spectral properties of  $\mathscr{H}$ and some closed quipus are achieved. They have been selected in view of the proof of Theorem~\ref{main}, given in Section~5.

\begin{figure}
	  	\centering
	  	
  \begin{tikzpicture}[
vertex3_style/.style={fill,circle,inner sep=0pt,minimum size=4pt},
edge_style/.style={ultra thick, black, drop shadow={opacity=0.1}},
edge_style/.style={ultra thick, black, drop shadow={opacity=0.1}},
nonterminal/.style={
rectangle,
minimum size=2mm,
thin,
draw=black,
top color=white, 
bottom color=white!50!white!50, 
font=\itshape
},
]

 \node[nonterminal] (Z) at (-1,2.7) {\footnotesize$P^{n_1,\dots, n_t}_{m_1,\dots,m_t;r}$};
 \node[vertex3_style]   (X1) at (-0.4,0)  [label=below: \footnotesize$0$]{};
\node[vertex3_style]   (X2) at (0.6,0)  [label=below: \footnotesize$m_1$] {};
\node[vertex3_style]   (X3) at (1.6,0)   [label=below: \footnotesize$m_2$]{};
\node[vertex3_style]   (X4) at (2.6,0)  [label=below: \footnotesize$m_t$] {};
\node[vertex3_style]   (X5) at (3.6,0)   [label=below: \footnotesize$r-1$]{};
\path [thick, dotted] (X1) edge (X5);
\node[vertex3_style]   (Y2) at (0.6,.5) {};
\node[vertex3_style]   (Y3) at (1.6,.5) {};
\node[vertex3_style]   (Y4) at (2.6,.5) {};
\node[vertex3_style]   (Z2) at (0.6,1.8) {};
\node[vertex3_style]   (Z3) at (1.6,1.8) {};
\node[vertex3_style]   (Z4) at (2.6,1.8) {};
 \path [thick] (X2) edge (Y2);
 \path [thick] (X3) edge (Y3);
 \path [thick] (X4) edge (Y4);
\path [thick, dotted] (Z2) edge (Y2);
 \path [thick, dotted] (Z3) edge (Y3);
 \path [thick, dotted] (Z4) edge (Y4);
\draw [decorate,decoration={brace,amplitude=7pt,mirror},xshift=0.3pt,yshift=0pt](0.45,1.9)--(0.45,0.4) node[black,midway,xshift=-0.45cm] {\footnotesize $n_1$};
\draw [decorate,decoration={brace,amplitude=7pt,mirror},xshift=0.3pt,yshift=0pt](1.45,1.9)--(1.45,0.4) node[black,midway,xshift=-0.45cm] {\footnotesize $n_2$};
 \draw [decorate,decoration={brace,amplitude=7pt,mirror},xshift=0.3pt,yshift=0pt](2.75,.4)--(2.75,1.9) node[black,midway,xshift=0.45cm] {\footnotesize $n_t$};
	\begin{scope}[xshift=8cm]
 \node[nonterminal] (Z2) at (-2.5,2.7) {\footnotesize$C^{0,n_2,\dots, n_t}_{m_1,\dots,m_t;r}$};
  \def\xmax{2.2}
  \def\ul{0.6}
  \def\R{.8}
  \def\ang{43}
  \coordinate (O) at (0,0.8);
  \coordinate (X) at (\xmax,0);
  \coordinate (R) at (\ang:\R);
 \draw[thick, dotted] (O) circle (\R);
\node[vertex3_style]   (T3) at (0,1.6)  [label=below: \footnotesize$0$]{};
\node[vertex3_style]   (T5) at (-0.8,.8) {};
\node at (1,0.23) {\footnotesize$m_2$};
\node at (-.4,.5) {\footnotesize$m_t$};
\node[vertex3_style]   (T8) at (0.56,0.24) {};
\node[vertex3_style]   (U3) at (0,2.1) {};
\node[vertex3_style]   (V3) at (0,2.9) {};
\path [thick] (T3) edge (U3);
\path [thick, dotted] (U3) edge (V3);
\draw [decorate,decoration={brace,amplitude=6pt,mirror},xshift=0.3pt,yshift=0pt](-0.15,3)--(-0.15,2) node[black,midway,xshift=-0.45cm] {\footnotesize $n_1$};
\node[vertex3_style]   (U8) at (0.96,-.16) {};
\node[vertex3_style]   (V8) at (1.56,-0.76) {};
\path [thick] (T8) edge (U8);
\path [thick, dotted] (U8) edge (V8);
\node[vertex3_style]   (U5) at (-1.4,0.8) {};
\node[vertex3_style]   (V5) at (-2.4,0.8) {};
\path [thick] (T5) edge (U5);
\path [thick, dotted] (U5) edge (V5);
\draw [decorate,decoration={brace,amplitude=6pt,mirror},xshift=0.3pt,yshift=0pt](-1.3,.95)--(-2.5,.95) node[black,midway,yshift=0.45cm] {\footnotesize $n_t$};
\draw [decorate,decoration={brace,amplitude=6pt,mirror},xshift=0.3pt,yshift=0pt](0.8,-.2)--(1.5,-.9) node[black,midway,xshift=-.4cm, yshift=-0.3cm] {\footnotesize $n_2$};
\begin{scope}[xshift=4.5cm, yshift=-.2cm]
\node[vertex3_style]   (H1) at (0,2.8) {};
\node[vertex3_style]   (H2) at (0,2) {};
\node[vertex3_style]   (H3) at (0,1.2) {};
\node[vertex3_style]   (H4) at (0,0) {};
\node[vertex3_style]   (G2) at (-.8,2) {};
\node[vertex3_style]   (I2) at (.8,2) {};
\path [thick] (H1) edge (H3);
\path [thick, dotted] (H3) edge (H4);
\path [thick] (G2) edge (I2);
\end{scope}
\end{scope}

\end{tikzpicture}
\caption{ \label{Fig2}  \small Open quipus, closed quipus and daggers}
\end{figure}
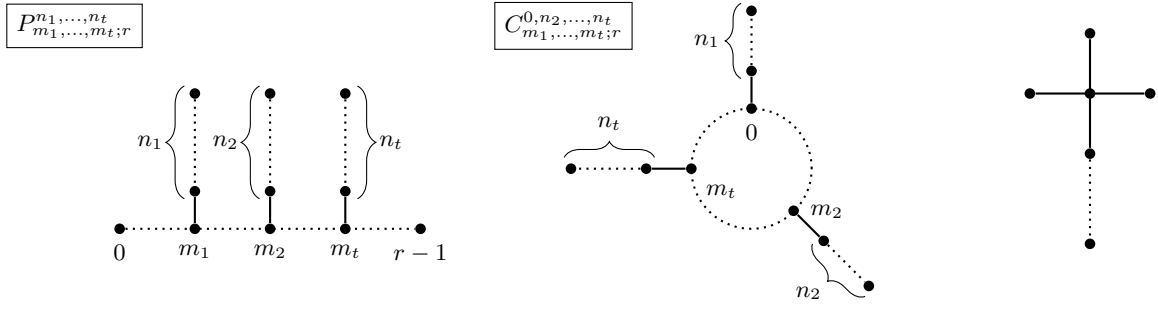
\section{Preliminaries}\label{S2}

\begin{lem}\label{path}{\rm \cite[p. 73]{cve1}} For $n \in \mathbb N$, the following equalities hold.
$$\phi(P_n,\lambda)  = \prod_{i=1}^n \left(\lambda - 2\cos{\frac{\pi i}
{n+1}}\right) = \sum_{r=0}^{\lfloor \frac{n}{2}\rfloor}(-1)^r{n-r \choose r}\lambda^{n-2r}$$
and
$$ \phi(C_n,\lambda)=  \prod_{i=0}^{n-1} \left(\lambda - 2\cos{\frac{2\pi i}
{n}}\right)= -2 +   \sum_{r=0}^{\lfloor \frac{n}{2}\rfloor}(-1)^r \frac{n}{n-r}{n-r \choose r}\lambda^{n-2r}.$$
\end{lem}

\begin{lem}\label{pathHvD} {\rm \cite[Proposition 1]{dam}} The path $P_n$ is DS for every positive integer $n$.
\end{lem}

In order to prove that a graph to be DS, it is useful to find as many as possible algebraic invariants for cospectral graphs which are summarized in the following proposition (for a proof see, for instance, \cite[Lemma 4]{dam}).

\begin{prop}\label{invariant}
Let $G$ and $H$ be two cospectral graphs. Then,\smallskip

\noindent $\mathrm{(i)}$ $\nu_G = \nu_H$ and $\varepsilon_G = \varepsilon_H$.

\noindent $\mathrm{(ii)}$ $G$ is bipartite if and only if $H$ is bipartite.

\noindent $\mathrm{(iii)}$ $G$ is $k$-regular if and only if $H$ is $k$-regular.

\noindent $\mathrm{(iv)}$ $G$ is $k$-regular with girth $g$ if and only if $H$ is
$k$-regular with girth $g$.

\noindent $\mathrm{(v)}$ $G$ and $H$ have the same number of closed walks of
any fixed length.
\end{prop}

\begin{lem}\label{smith}{\rm \cite{cve1,smith}} Let $\mathscr{G}_{<2}$ be the set of connected graphs
 whose index is less than $2$. Then,
$$ \mathscr{G}_{<2}= \left\{P_n(n \geq 1), \, P_{1;n-1}^1(n \geq 4)\} \cup
\{P_{2;k-1}^1 \mid k = 5,6,7 \right\}.$$\end{lem}

\begin{lem}\label{smith1}{\rm \cite{cve1,smith}}   Let $\mathscr{G}_{2}$ be the set of connected graphs
 whose index is $2$. Then,
$$\mathscr{G}_2 = \left\{C_n \; (n \geq 3), \, P_{1,1;n-2}^{1,n-4} \; (n \geq 6),\,  K_{1,4}, \, P_{2;5}^2,\,
P_{1;8}^2, \, P_{1;7}^3 \right\}.$$\end{lem}
Throughout the paper we denote by $\mathfrak h$ the number $\sqrt{2+\sqrt{5}}$, known in literature as the (adjacency)-Hoffman limit value.

\begin{lem}\label{bro}{\rm \cite{bro,cvek}} The set $\mathscr{G}_{\left(2,\mathfrak h\right)}$ of  connected graphs
 whose index is in the interval
$\left(2,\mathfrak h\right)$ only contains T-shape and H-shape graphs. More precisely, $\mathscr{G}_{\left(2,\mathfrak h\right)} = \mathcal T \cup \mathcal H$, where
$${\mathcal T} = \left\{ P_{c;4}^1 \; \middle\vert \;  c > 5 \right\}  \cup \left\{ P_{c;b+2}^1  \; \middle\vert \;  b>2, \;  c > 3 \right\} \cup  \left\{ P_{c;5}^2  \; \middle\vert \;  c > 2 \right\} \cup \left\{ P_{3;6}^2 \right\}$$
and
$$ {\mathcal H} = \left\{ P_{1,1;5}^{1,2}, \;\,  P_{1,1;9}^{2,6}, \;\, P_{1,1;11}^{2,7}, \;\,  P_{1,1;14}^{3,10}, \;\,  P_{1,1;16}^{3,11} \right\} \cup \left\{  P_{1,1;a+b+c+1}^{a,a+b} \; \Big\vert \;   a > 0, c >
0, b \geq b^{*}(a,c)  \right\},$$
with $b^{*}(a, c) = \begin{cases} \, c \qquad \qquad \;\, \text {for $a=1$},\\
                                                  \, c+3 \qquad \;\;\, \text {for $a=2$},\\
                                                  \, a+c+2 \quad \text{for $a>2$}.
\end{cases}$
\end{lem}

\begin{lem}\label{woo}{\rm \cite[Theorem 1]{woo}}
The connected graphs with spectral radii in the interval $\left( \, {\mathfrak h}, \,  3\sqrt{2}/2 \, \right]$
are either open quipus or closed quipus or daggers (see Fig. \ref{Fig2}).
\end{lem}
We write $H \subseteq G$ (resp., $H \subset G$) if $H$ is a subgraph (resp., proper subgraph) of the graph $G$.

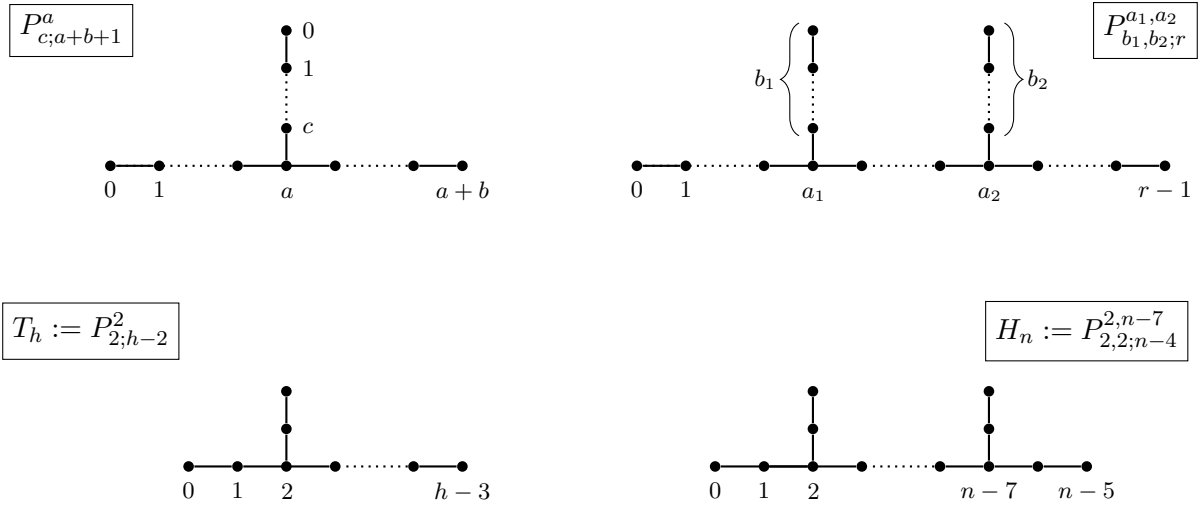
\begin{figure}
	  	\centering

  \begin{tikzpicture}[
vertex3_style/.style={fill,circle,inner sep=0pt,minimum size=4pt},
edge_style/.style={ultra thick, black, drop shadow={opacity=0.1}},
edge_style/.style={ultra thick, black, drop shadow={opacity=0.1}},
nonterminal/.style={
rectangle,
minimum size=2mm,
thin,
draw=black,
top color=white, 
bottom color=white!50!white!50, 
font=\itshape
},
 x=1.3cm,
]
 \node[vertex3_style]   (A1) at (0.2,0)  [label=below: \footnotesize$0$]{};
\node[vertex3_style]   (B1) at (0.7,0)  [label=below: \footnotesize$1$] {};
\node[vertex3_style]   (C1) at (1.5,0)  {};
\path [thick, dotted] (C1) edge (A1);
 \path [thick] (A1) edge (B1);
\node[vertex3_style]   (D1) at (2,0)  [label=below: \footnotesize$\;\;\;\;\,a\phantom{+b}$] {};
\node[vertex3_style]   (E1) at (2.5,0) {};
\node[vertex3_style]   (F1) at (3.3,0)   {};
\node[vertex3_style]   (G1) at (3.8,0)   [label=below: \footnotesize$a+b$] {};
 \path [thick] (C1) edge (E1);
 \path [thick, dotted] (F1) edge (E1);
 \path [thick] (F1) edge (G1);
 \node[vertex3_style]   (H1) at (2,0.5)  [label=right: \footnotesize$c$]{};
 \node[vertex3_style]   (I1) at (2,1.3)  [label=right: \footnotesize$1$]{};
 \node[vertex3_style]   (L1) at (2,1.8)  [label=right: \footnotesize$0$]{};
 \path [thick] (L1) edge (I1);
 \path [thick, dotted] (I1) edge (H1);
 \path [thick] (H1) edge (D1);
 \node[nonterminal] (Z) at (-.2,1.8) {$ P^a_{c;a+b+1}$};

\begin{scope}[xshift=7cm]
 \node[vertex3_style]   (A1) at (0.2,0)  [label=below: \footnotesize$0$]{};
\node[vertex3_style]   (B1) at (0.7,0)  [label=below: \footnotesize$1$] {};
\node[vertex3_style]   (C1) at (1.5,0)  {};
\path [thick, dotted] (C1) edge (A1);
 \path [thick] (A1) edge (B1);
\node[vertex3_style]   (D1) at (2,0)  [label=below: \footnotesize$\;\;\;\;\,a_1\phantom{+b}$] {};
\node[vertex3_style]   (E1) at (2.5,0) {};
\node[vertex3_style]   (F1) at (3.3,0)   {};
\node[vertex3_style]   (G1) at (3.8,0)   [label=below: \footnotesize$\;\;\;\;\,a_2\phantom{+b}$] {};
\node[vertex3_style]   (W1) at (4.3,0)  {};
\node[vertex3_style]   (W2) at (5.1,0)  {};
\node[vertex3_style]   (W3) at (5.6,0)  [label=below: \footnotesize$r-1$] {};
 \path [thick] (C1) edge (E1);
 \path [thick, dotted] (F1) edge (E1);
 \path [thick] (F1) edge (G1);
 \node[vertex3_style]   (H1) at (2,0.5) {};
 \node[vertex3_style]   (I1) at (2,1.3)  {};
 \node[vertex3_style]   (L1) at (2,1.8)  {};
 \path [thick] (L1) edge (I1);
 \path [thick, dotted] (I1) edge (H1);
 \path [thick] (H1) edge (D1);
 \node[nonterminal] (Z) at (5.4,1.8) {$ P^{a_1,a_2}_{b_1,b_2;r}$};
\draw [decorate,decoration={brace,amplitude=7pt,mirror},xshift=0.3pt,yshift=0pt](1.85,1.9)--(1.85,0.4) node[black,midway,xshift=-0.45cm] {\footnotesize $b_1$};
 \path [thick] (G1) edge (W1);
 \path [thick, dotted] (W2) edge (W1);
 \path [thick] (W2) edge (W3);
 \node[vertex3_style]   (H2) at (3.8,0.5) {};
 \node[vertex3_style]   (I2) at (3.8,1.3)  {};
 \node[vertex3_style]   (L2) at (3.8,1.8)  {};
 \path [thick] (L2) edge (I2);
 \path [thick, dotted] (I2) edge (H2);
 \path [thick] (H2) edge (G1);
\draw [decorate,decoration={brace,amplitude=7pt,mirror},xshift=0.3pt,yshift=0pt](3.95,0.4)--(3.95,1.9) node[black,midway,xshift=0.45cm] {\footnotesize $b_2$};
\end{scope}
 \begin{scope}[yshift=-4cm]

\node[vertex3_style]   (B1) at (1,0)  [label=below: \footnotesize$0$] {};
\node[vertex3_style]   (C1) at (1.5,0) [label=below: \footnotesize$1$]  {};
\path [thick] (C1) edge (B1);
\node[vertex3_style]   (D1) at (2,0)  [label=below: \footnotesize$\;\;\;\;\,2\phantom{+b}$] {};
\node[vertex3_style]   (E1) at (2.5,0) {};
\node[vertex3_style]   (F1) at (3.3,0)   {};
\node[vertex3_style]   (G1) at (3.8,0)   [label=below: \footnotesize$h-3$] {};
 \path [thick] (C1) edge (E1);
 \path [thick, dotted] (F1) edge (E1);
 \path [thick] (F1) edge (G1);
 \node[vertex3_style]   (H1) at (2,0.5) {};
 \node[vertex3_style]   (I1) at (2,1)  {};
 \path [thick] (I1) edge (H1);
 \path [thick] (H1) edge (D1);
 \node[nonterminal] (Z) at (0,1.8) {$ T_h:= P^2_{2;h-2}$};

\begin{scope}[xshift=7cm]
\node[vertex3_style]   (B1) at (1,0)  [label=below: \footnotesize$0$] {};
\node[vertex3_style]   (C1) at (1.5,0)[label=below: \footnotesize$1$]  {};
\node[vertex3_style]   (D1) at (2,0)  [label=below: \footnotesize$\;\;\;\;\,2\phantom{+b}$] {};
\path [thick] (D1) edge (B1);
\node[vertex3_style]   (E1) at (2.5,0) {};
\node[vertex3_style]   (F1) at (3.3,0)   {};
\node[vertex3_style]   (G1) at (3.8,0)   [label=below: \footnotesize$n-7$] {};
\node[vertex3_style]   (W1) at (4.3,0)  {};
\node[vertex3_style]   (W2) at (4.8,0) [label=below: \footnotesize$n-5$]  {};
 \path [thick] (C1) edge (E1);
 \path [thick, dotted] (F1) edge (E1);
 \path [thick] (F1) edge (G1);
 \node[vertex3_style]   (H1) at (2,0.5) {};
 \node[vertex3_style]   (I1) at (2,1)  {};
 \path [thick] (I1) edge (H1);
 \path [thick] (H1) edge (D1);
 \node[nonterminal] (Z) at (4.8,1.8) {$ H_n:=P^{2,n-7}_{2,2;n-4}$};
 \path [thick] (G1) edge (W1);
 \path [thick] (W2) edge (W1);
 \node[vertex3_style]   (H2) at (3.8,0.5) {};
 \node[vertex3_style]   (I2) at (3.8,1)  {};
\path [thick] (I2) edge (H2);
 \path [thick] (H2) edge (G1);
\end{scope}

\end{scope}

\end{tikzpicture}
\caption{ \label{Fig1}  \small T-shape and H-shape graphs}
\end{figure}

\begin{lem}\label{proper}{\rm \cite[Theorems 0.6 and 0.7]{cve1}}
Let $G$ be a connected graph and $H \subset G$. Then $\rho(H) < \rho(G)$.
\end{lem}
For $v \in V_G$, let $G-v$ be the graph obtained from $G$ by deleting $v$ and its incident edges.
\begin{lem}\label{interlace}{\rm \cite[Theorem 0.10]{cve1}}
Let $\lambda_1 \geqslant \lambda_2 \geqslant \cdots \geqslant \lambda_n$ and $\mu_1
\geqslant \mu_2 \geqslant \cdots \geqslant \mu_{n-1}$ be the eigenvalues of the graphs $G$ and $G-v$ respectively.
Then, $\lambda_1 \geqslant \mu_1
\geqslant \lambda_2 \geqslant \mu_2 \geqslant \cdots \geqslant \mu_{n-1} \geqslant
\lambda_n$.
\end{lem}

Denoted by $d(v)$ the vertex degree of $v \in V_G$ in $G$, an {\it internal path} of $G$ is a (possibly closed) walk $v_1 \dots v_k$ such the $\min \{d(v_1), d(v_k)\} \geqslant 3$ and $d(v_i) = 2$ for $2 \leqslant i \leqslant k-1$.
We also recall that {\it subdividing an edge $uv \in E_G$} means  inserting a new vertex $z$ in $V_G$ and replacing $uv$ with $uz$ and $zw$.
The two parts of the following lemma come from Lemma \ref{proper} and \cite{hof} respectively.
\begin{lem}\label{internalpath}
Let $uv$ be an edge of a connected graph $G$ and let $G_{uv}$ be the graph
obtained from $G$ by subdividing the edge $uv \in E_G$. Then,\smallskip

\noindent $\mathrm{(i)}$ if $uv$ is not in an internal path of $G$ and $G$ is not a cycle, then
$\rho(G_{uv}) > \rho(G)$;\smallskip

\noindent
$\mathrm{(ii)}$ if $uv$ belongs to an internal path of $G$ and $G \not \in \left\{
P_{1,1;n-2}^{1,n-4} \mid  n \geqslant 6 \right\}$, then $\rho(G_{uv}) < \rho(G)$.
\end{lem}

Let $v$ be a vertex of a graph $G$. As it is usual, we denote by $N_G(v)$ the {\it neighborhhod } of $v$, i.e.\ the set $\{ w \in V_G \mid vw \in E_G \}$.
\begin{lem}\label{formula1}{\rm \cite[Theorems 2 and 3]{Sch}} Let $\mathscr{C}(v)$ (resp., $\mathscr{C}(e)$) be the set
of all cycles of a graph  $G$ containing the vertex $v \in V_G$ (resp., the edge $e=uw \in E_G$).  The following identities of polynomials

\begin{equation}\label{schwenky1}
\phi(G)= \lambda \phi(G-v) - \sum_{v' \in N_G(v)}\phi(G-v-v')-2\sum_{C \in \mathscr{C}(v)}\phi(G-V(C))\end{equation}
and
\begin{equation}\label{schwenky2}
\phi(G) = \phi(G-e) - \phi(G-u-w) -2\sum_{C \in
\mathscr{C}(e)}\phi(G-V(C))
\end{equation}
hold for every $v \in V_G$  and for every $e=uw \in E_G$ (note that $\phi(H)=1$ if $H$ is the null graph $K_0$).
\end{lem}

Equations \eqref{schwenky1} and \eqref{schwenky2} are usually called {\it Schwenk formul\ae} and could be used as an alternative to direct calculations to prove the following identities:
\begin{equation}\label{035}\phi(H_{10}) = \phi(P_2 \cup C_{1,1;6}^{0,3});\quad  \phi(H_{13}) = \phi(P_{2;4}^1 \cup C_{1;6}^0);\quad
\phi(H_{15}) = \phi(P_{1,2,6}^{1,3} \cup C_6),
\end{equation}
(the $H_n$'s have been defined in \eqref{accan}).

Let $k$ be a positive integer. We recall that a $k$-matching in a graph $G$ is a set of $k$ independent edges.
We denote by $M_k(G)$ the number of $k$-matchings in a graph $G$, and by $\Delta_G$ its maximum vertex degree. The following result follows from the Sachs's Coefficient Theorem for characteristic polynomials of graphs (see \cite[Theorem 1.3]{cve1}).
\begin{prop}\label{ahah}
Let $G$ and $H$ be two cospectral graphs. \smallskip

\noindent $\mathrm{(i)}$
 If neither $G$ nor $H$ contain quadrangles as subgraphs, then $M_2(G) = M_2(H)$. \smallskip

\noindent$\mathrm{(ii)}$  If neither $G$ nor $H$ contain quadrangles or hexagons as subgraphs, then $M_3(G) = M_3(H)$.
\end{prop}

\begin{lem}\label{Match}
Let $G$ be a graph with $N$ triangles and degree sequence
$(d_1,d_2,\dots,d_{\nu_G})$. \smallskip

\noindent{\rm (i)}
The number of 2-matchings in $G$ is
$\displaystyle M_2(G) = \binom{\varepsilon_G}{2} - \sum_{i=1}^{\nu_G}\binom{d_i}{2}.$ \smallskip

\noindent {\rm (ii)}\label{2match} If $T$ is a tree with maximal degree $\Delta_T = 3$, then $ \displaystyle M_2(T) = \frac{\nu_T^2-5\nu_T}{2}+3-k_T,$
where $k_T$ is the number of vertices of degree $3$ in $V_T$.

If, instead, $G$ is a closed quipu, then  $ \displaystyle M_2(G) = \frac{\nu_G^2- 3\nu_G}{2} - k$. \smallskip

\noindent{\rm (iii)}\label{3match}
The number of 3-matchings in $G$ is
$$M_3(G) = \binom{\varepsilon_G}{3}
- (\varepsilon_G - 2)\sum_{i=1}^{\nu_G}\binom{d_i}{2} +
2\sum_i\binom{d_i}{3} + \sum_{ij \in E(G)}(d_i - 1)(d_j - 1) - N.$$
\end{lem}
\begin{proof} Part (i) is elementary: the number $M_2(G)$ is obtained by subtracting the number of pairs of dependent edges from the total number of pairs of edges. Part (ii) is essentially Lemma 2.10 in \cite{wang1}. For Part (iii), see \cite[Theorem 1]{far}.
\end{proof}

The several types of H-shape trees are depicted in Fig. \ref{Fig3}, where dotted lines correspond to paths with at least two edges. In other words, the set of the H-shape trees is the disjoint union $\bigsqcup_{i=1}^{12} {\mathcal T_i}$, where ${\mathcal T}_1 = \left\{  P_{1,1;a}^{1,a-2} \;  \big\vert \; a \geqslant 5   \right\}$, ${\mathcal T}_2 = \left\{  P_{1,1;4}^{1,2}\right\}$, ${\mathcal T}_3= \left\{  P_{1,1;a}^{1,b} \;  \big\vert \; a \geqslant 6, \; 2<b<a-2   \right\}$. and so on.

 By the above lemma and calculations, we can classify the H-shape trees by means of the number of their 3-matchings.

\begin{cor}\label{3match1} Let $T$ be an H-shape tree with $n$ vertices, and let
$f(n) = (n^3 - 12n^2 + 35n)/6$. Then,
$$M_3(T) =
\begin{cases}
f(n), & \mbox{for} \quad T \in {\mathcal T}_1;\\
f(n) + 1, & \mbox{for}  \quad T \in {\mathcal T}_2 \cup {\mathcal T}_3;\\
f(n) + 2, & \mbox{for}   \quad T \in {\mathcal T}_4 \cup {\mathcal T}_5  \cup {\mathcal T}_6; \\
f(n) + 3, & \mbox{for}   \quad T \in {\mathcal T}_7 \cup {\mathcal T}_8  \cup {\mathcal T}_{9}; \\
f(n) + 4, & \mbox{for}   \quad T \in {\mathcal T}_{10} \cup {\mathcal T}_{11}; \\
f(n) + 5, & \mbox{for}   \quad T \in {\mathcal T}_{12}.
\end{cases}
$$
\end{cor}
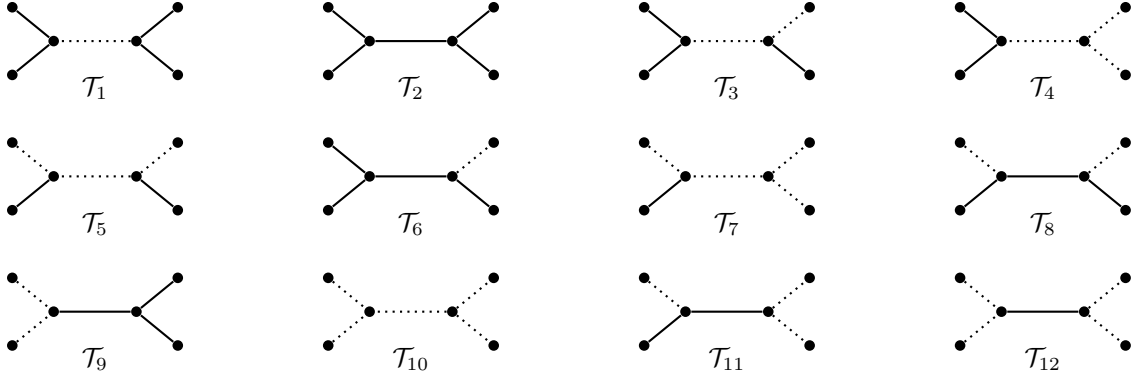
\begin{figure}
	  	\centering

  \begin{tikzpicture}[
vertex3_style/.style={fill,circle,inner sep=0pt,minimum size=4pt},
edge_style/.style={ultra thick, black, drop shadow={opacity=0.1}},
edge_style/.style={ultra thick, black, drop shadow={opacity=0.1}},
nonterminal/.style={
rectangle,
minimum size=2mm,
thin,
draw=black,
top color=white, 
bottom color=white!50!white!50, 
font=\itshape
},
 x=1.1cm,
y=.9cm
]

\node[vertex3_style]   (A1) at (-0.5,-0.5) {};
\node[vertex3_style]   (B1) at (-.5,.5) {};
\node[vertex3_style]   (C1) at (0,0) {};
\node[vertex3_style]   (D1) at (1,0) {};
\node[vertex3_style]   (E1) at (1.5,.5) {};
\node[vertex3_style]   (F1) at (1.5,-.5) {};
\path [thick] (A1) edge (C1);
\path [thick] (B1) edge (C1);
\path [thick, dotted] (D1) edge (C1);
\path [thick] (D1) edge (E1);
\path [thick] (D1) edge (F1);
\node at (.5, -0.7) {\small$\mathcal T_1$};
\begin{scope}[xshift=4.2cm]
\node[vertex3_style]   (A1) at (-0.5,-0.5) {};
\node[vertex3_style]   (B1) at (-.5,.5) {};
\node[vertex3_style]   (C1) at (0,0) {};
\node[vertex3_style]   (D1) at (1,0) {};
\node[vertex3_style]   (E1) at (1.5,.5) {};
\node[vertex3_style]   (F1) at (1.5,-.5) {};
\path [thick] (A1) edge (C1);
\path [thick] (B1) edge (C1);
\path [thick] (D1) edge (C1);
\path [thick] (D1) edge (E1);
\path [thick] (D1) edge (F1);
\node at (.5, -0.7) {\small$\mathcal T_2$};
\end{scope}
\begin{scope}[xshift=8.4cm]
\node[vertex3_style]   (A1) at (-0.5,-0.5) {};
\node[vertex3_style]   (B1) at (-.5,.5) {};
\node[vertex3_style]   (C1) at (0,0) {};
\node[vertex3_style]   (D1) at (1,0) {};
\node[vertex3_style]   (E1) at (1.5,.5) {};
\node[vertex3_style]   (F1) at (1.5,-.5) {};
\path [thick] (A1) edge (C1);
\path [thick] (B1) edge (C1);
\path [thick, dotted] (D1) edge (C1);
\path [thick, dotted] (D1) edge (E1);
\path [thick] (D1) edge (F1);
\node at (.5, -0.7) {\small$\mathcal T_3$};
\end{scope}
\begin{scope}[xshift=12.6cm]
\node[vertex3_style]   (A1) at (-0.5,-0.5) {};
\node[vertex3_style]   (B1) at (-.5,.5) {};
\node[vertex3_style]   (C1) at (0,0) {};
\node[vertex3_style]   (D1) at (1,0) {};
\node[vertex3_style]   (E1) at (1.5,.5) {};
\node[vertex3_style]   (F1) at (1.5,-.5) {};
\path [thick] (A1) edge (C1);
\path [thick] (B1) edge (C1);
\path [thick, dotted] (D1) edge (C1);
\path [thick, dotted] (D1) edge (E1);
\path [thick, dotted] (D1) edge (F1);
\node at (.5, -0.7) {\small$\mathcal T_4$};
\end{scope}
\begin{scope}[yshift=-1.8cm]
\node[vertex3_style]   (A1) at (-0.5,-0.5) {};
\node[vertex3_style]   (B1) at (-.5,.5) {};
\node[vertex3_style]   (C1) at (0,0) {};
\node[vertex3_style]   (D1) at (1,0) {};
\node[vertex3_style]   (E1) at (1.5,.5) {};
\node[vertex3_style]   (F1) at (1.5,-.5) {};
\path [thick] (A1) edge (C1);
\path [thick, dotted] (B1) edge (C1);
\path [thick, dotted] (D1) edge (C1);
\path [thick, dotted] (D1) edge (E1);
\path [thick] (D1) edge (F1);
\node at (.5, -0.7) {\small$\mathcal T_5$};
\end{scope}
\begin{scope}[xshift=4.2cm, yshift=-1.8cm]
\node[vertex3_style]   (A1) at (-0.5,-0.5) {};
\node[vertex3_style]   (B1) at (-.5,.5) {};
\node[vertex3_style]   (C1) at (0,0) {};
\node[vertex3_style]   (D1) at (1,0) {};
\node[vertex3_style]   (E1) at (1.5,.5) {};
\node[vertex3_style]   (F1) at (1.5,-.5) {};
\path [thick] (A1) edge (C1);
\path [thick] (B1) edge (C1);
\path [thick] (D1) edge (C1);
\path [thick, dotted] (D1) edge (E1);
\path [thick] (D1) edge (F1);
\node at (.5, -0.7) {\small$\mathcal T_6$};
\end{scope}
\begin{scope}[xshift=8.4cm, yshift=-1.8cm]
\node[vertex3_style]   (A1) at (-0.5,-0.5) {};
\node[vertex3_style]   (B1) at (-.5,.5) {};
\node[vertex3_style]   (C1) at (0,0) {};
\node[vertex3_style]   (D1) at (1,0) {};
\node[vertex3_style]   (E1) at (1.5,.5) {};
\node[vertex3_style]   (F1) at (1.5,-.5) {};
\path [thick, dotted] (B1) edge (C1);
\path [thick] (A1) edge (C1);
\path [thick, dotted] (D1) edge (C1);
\path [thick, dotted] (D1) edge (E1);
\path [thick, dotted] (D1) edge (F1);
\node at (.5, -0.7) {\small$\mathcal T_7$};
\end{scope}
\begin{scope}[xshift=12.6cm, yshift=-1.8cm]
\node[vertex3_style]   (A1) at (-0.5,-0.5) {};
\node[vertex3_style]   (B1) at (-.5,.5) {};
\node[vertex3_style]   (C1) at (0,0) {};
\node[vertex3_style]   (D1) at (1,0) {};
\node[vertex3_style]   (E1) at (1.5,.5) {};
\node[vertex3_style]   (F1) at (1.5,-.5) {};
\path [thick] (A1) edge (C1);
\path [thick, dotted] (B1) edge (C1);
\path [thick] (D1) edge (C1);
\path [thick, dotted] (D1) edge (E1);
\path [thick] (D1) edge (F1);
\node at (.5, -0.7) {\small$\mathcal T_8$};
\end{scope}
\begin{scope}[yshift=-3.6cm]
\node[vertex3_style]   (A1) at (-0.5,-0.5) {};
\node[vertex3_style]   (B1) at (-.5,.5) {};
\node[vertex3_style]   (C1) at (0,0) {};
\node[vertex3_style]   (D1) at (1,0) {};
\node[vertex3_style]   (E1) at (1.5,.5) {};
\node[vertex3_style]   (F1) at (1.5,-.5) {};
\path [thick, dotted] (A1) edge (C1);
\path [thick, dotted] (B1) edge (C1);
\path [thick] (D1) edge (C1);
\path [thick] (D1) edge (E1);
\path [thick] (D1) edge (F1);
\node at (.5, -0.7) {\small$\mathcal T_9$};
\end{scope}
\begin{scope}[xshift=4.2cm, yshift=-3.6cm]
\node[vertex3_style]   (A1) at (-0.5,-0.5) {};
\node[vertex3_style]   (B1) at (-.5,.5) {};
\node[vertex3_style]   (C1) at (0,0) {};
\node[vertex3_style]   (D1) at (1,0) {};
\node[vertex3_style]   (E1) at (1.5,.5) {};
\node[vertex3_style]   (F1) at (1.5,-.5) {};
\path [thick, dotted] (A1) edge (C1);
\path [thick, dotted] (B1) edge (C1);
\path [thick, dotted] (D1) edge (C1);
\path [thick, dotted] (D1) edge (E1);
\path [thick, dotted] (D1) edge (F1);
\node at (.5, -0.7) {\small$\mathcal T_{10}$};
\end{scope}
\begin{scope}[xshift=8.4cm, yshift=-3.6cm]
\node[vertex3_style]   (A1) at (-0.5,-0.5) {};
\node[vertex3_style]   (B1) at (-.5,.5) {};
\node[vertex3_style]   (C1) at (0,0) {};
\node[vertex3_style]   (D1) at (1,0) {};
\node[vertex3_style]   (E1) at (1.5,.5) {};
\node[vertex3_style]   (F1) at (1.5,-.5) {};
\path [thick] (A1) edge (C1);
\path [thick, dotted] (B1) edge (C1);
\path [thick] (D1) edge (C1);
\path [thick, dotted] (D1) edge (E1);
\path [thick, dotted] (D1) edge (F1);
\node at (.5, -0.7) {\small$\mathcal T_{11}$};
\end{scope}
\begin{scope}[xshift=12.6cm, yshift=-3.6cm]
\node[vertex3_style]   (A1) at (-0.5,-0.5) {};
\node[vertex3_style]   (B1) at (-.5,.5) {};
\node[vertex3_style]   (C1) at (0,0) {};
\node[vertex3_style]   (D1) at (1,0) {};
\node[vertex3_style]   (E1) at (1.5,.5) {};
\node[vertex3_style]   (F1) at (1.5,-.5) {};
\path [thick, dotted] (A1) edge (C1);
\path [thick, dotted] (B1) edge (C1);
\path [thick] (D1) edge (C1);
\path [thick, dotted] (D1) edge (E1);
\path [thick, dotted] (D1) edge (F1);
\node at (.5, -0.7) {\small$\mathcal T_{12}$};
\end{scope}
\end{tikzpicture}
\caption{ \label{Fig3}  \small The several types of H-shape trees}
\end{figure}

At the end of this section, we consider the {\it lowest term} $\vartheta(G)$ of a graph $G$, i.e.\  the non-zero monomial of lowest degree in its characteristic polynomial. Clearly, $\vartheta(G)$ is an algebraic invariant shared by all graphs which are cospectral to $G$. The degree of $\vartheta(G)$ is also equal to the {\it nullity} $\eta(G)$, i.e.\ the multiplicity of $0$ in ${\rm sp}(G)$. The proofs of the following three propositions are postponed to Appendix A.

\begin{prop}\label{lollipop-low}
For each pair $(s,t)$ of positive integers,
$$\vartheta(C_{4s-2\epsilon_1;4t+2\epsilon_2}^0) =
\begin{cases}
-4t(s+t)\lambda^2, & \mbox{if } \quad (\epsilon_1,\epsilon_2) = (0,0); \\
-4,  & \mbox{if } \quad (\epsilon_1,\epsilon_2) = (0,1);\\
2t(2s+2t-1)\lambda^2, & \mbox{if } \quad (\epsilon_1,\epsilon_2) = (1,0);\\
\phantom{-}4, & \mbox{if } \quad (\epsilon_1,\epsilon_2) = (1,1).
\end{cases}
$$
\end{prop}
The next proposition gives the lowest term of the graphs in the family
\begin{equation}\label{tienne}  \mathscr{T}= \{T_n \}_{n \geqslant 6} \qquad \text{where $T_n:= P_{2;n-2}^2$.}
\end{equation}
\begin{prop}\label{T22c-low} Let $\mathscr{T}$ be the family of T-shape graphs defined in \eqref{tienne} and,
for each  integer $n\geqslant 6$, let
let $s$ be the integer such that $n=4s+\epsilon$ with $\epsilon \in \{0,1,2,3\}$. Then, for each $T_n \in \mathscr{T}$,
\begin{equation}\label{tetateta}
\vartheta(T_n) =  \begin{cases}
\phantom{-}1, &\mbox{if{\;}  $\epsilon= 0$;}\\[.5mm]
\phantom{-}(2s+1)\lambda, &\mbox{if{\;}  $\epsilon= 1$;}\\[.5mm]
-1,\phantom{nullnull} &\mbox{if{\;} $\epsilon= 2$;}\\[.5mm]
-2(s+1)\lambda, &\mbox{if{\;}  $\epsilon= 3$;}
\end{cases}
\end{equation}
\end{prop}

\begin{prop}\label{DK-low}
For each integer $n\geqslant 10$, let $H_n$ be the H-shape graph defined in \eqref{accan}, and let $s$ be the integer such that $n=4s+\epsilon$ with $\epsilon \in \{0,1,2,3\}$. Then,
\begin{equation}\label{DKDK}
\vartheta(H_n) =\begin{cases}
\phantom{-}1, &\mbox{if{\;} $n = 4s$;}\\[.5mm]
\phantom{-}(2s+1)\lambda, &\mbox{if{\;}  $n = 4s+1$;}\\[.5mm]
-1,&\mbox{if{\;}  $n = 4s+2$;}\\[.5mm]
-(2s+2)\lambda, &\mbox{if{\;}  $n = 4s+3$.}\\[.3mm]
\end{cases}
\end{equation}
\end{prop}

\section{Divisibility between graphs}\label{S3}

Despite its manifest applications in several areas of spectral graph theory, divisibility between graphs  has not yet attracted the attention it deserves. In fact, literature on the subject is still scarce. We can only mention the paper  \cite{gha1} by Ghareghani et al., in which the divisibility of a starlike tree by a path is investigated. Here, we   establish a necessary
and sufficient condition for a larger class of graphs to be divisible by a path.
\begin{Def} {\rm A graph $G$ is said to be recursive if it belongs to a sequence of graphs $\{ G_n\}_{n \geqslant h}$ such that $\nu_{G_n}=n$ and
\begin{equation}\label{recu}\phi(G_{n+2})= \lambda     \phi(G_{n+1}) -  \phi(G_{n}).
\end{equation}}
\end{Def}

From  \eqref{schwenky1} it is easily seen that the existence of a pendant vertex  $v \in V(G)$ is a sufficient condition for $G$ to be recursive. If fact, if $u$ is the only neighbour of $v$ in $G$, we can consider the sequence $\{F_n\}_{n\geqslant \nu_G-2}$
such that $F_{\nu_G-2}=G-u-v$,  $F_{\nu_G-1}=G-v$, $F_{\nu_G}=G$, and $F_n$
for $n> \nu_G$ is obtained from $G$ by the coalescence of $v$ with an ending vertex of the path $P_{n-\nu_G+1}$. Yet, as  Proposition~\ref{p22n-4} in this section will show, not all families of recursive graphs are obtained by
attaching to a fixed graph a path of arbitrarily large length.
\begin{lem}\label{pathprime}
For each $n>1$ the polynomials $\phi(P_n)$ and $\phi(P_{n-1})$ are coprime.
\end{lem}
\begin{proof} Lemma \ref{formula1} allows to retrieve quite easily the well-known recursive formula
\begin{equation}\label{path-poly}
\phi(P_n) = \lambda\phi(P_{n-1}) - \phi(P_{n-2}).
\end{equation}
which holds for all $n\geqslant 2$. Since $\phi(P_0)=1$ and $\phi(P_1)=\lambda$ are coprime, the result follows from \eqref{path-poly} by using a  straightforward inductive argument on $n$.
\end{proof}
\begin{lem}\label{mink} Let $\{ g_n(\lambda)\}_{n \geqslant 0}$ be a sequence of polynomials, whose elements satisfy
\begin{equation}\label{recupoly} g_{n+2}(\lambda)= \lambda     g_{n+1}(\lambda) -  g_{n}(\lambda) \qquad
\text{for all $n \geqslant 0$.}\end{equation}
 Then,\smallskip

\noindent $\mathrm{(i)}$
$g_n(\lambda) = \phi(P_k)\, g_{n-k}(\lambda) - \phi(P_{k-1})\, g_{n-k-1}(\lambda)$ for $1 \leqslant k \leqslant n-1$;\smallskip

\noindent $\mathrm{(ii)}$
for each positive $i$, $\phi(P_n) \mid g_{n+1+i}(\lambda)$ if and only if $\phi(P_n) \mid
g_i (\lambda)$.
\end{lem}
\begin{proof} For sake of conciseness we set $g_i:=g_i(\lambda)$.
We prove (i) arguing inductively on $k$. Since $\phi(P_0) = 1$ and $\phi(P_1) = \lambda$, for $k=1$ (i) immediately comes from \eqref{recupoly}.
 Assume now that (i) holds for $k \leqslant l-1 <n-2$.
By \eqref{recu} and the induction hypothesis,
\begin{align*}
g_n  &= \lambda g_{n-1} - g_{n-2}\\[.2em]
&= \lambda \left(\phi(P_{l-1})g_{n-l} - \phi(P_{l-2}) g_{n-l-1}\right)  - \left(\phi(P_{l-2})g_{n-l}  - \phi(P_{l-3})
g_{n-l-1}\right)\\[.2em]
&= \left(\lambda\phi(P_{l-1}) - \phi(P_{l-2}) \right) g_{n-l} -
\left(\lambda\phi(P_{l-2}) - \phi(P_{l-3}) \right) g_{n-l-1} \\[.2em]
&= \phi(P_l)g_{n-l} - \phi(P_{l-1})g_{n-l-1},
\end{align*}
which finishes the proof of (i). Note that (i) implies the identity of polynomials
$$g_{n+1+i}=
\phi(P_{n+1})g_i - \phi(P_n)g_{i-1} \qquad (i \geqslant 2),$$ from which (ii) follows immediately since, by Lemma~\ref{pathprime}, $\phi(P_n)$ and $\phi(P_{n+1})$ are coprime.
\end{proof}
\begin{thm}\label{divi1}
Let $\{ G_n\}_{n \geqslant h}$ be a sequence of recursive graphs, whose elements satisfy \eqref{recu}, and let $g_n := \phi(G_n,\lambda)$. Then,\smallskip

\noindent $\mathrm{(i)}$
$g_n = \phi(P_k)\, g_{n-k} - \phi(P_{k-1})\, g_{n-k-1}$ for $1 \leqslant k \leqslant n-1-h$;\smallskip

\noindent $\mathrm{(ii)}$
for each integer $i\geqslant 2$, $\phi(P_n) \mid g_{n+1+i}$ if and only if $\phi(P_n) \mid
g_i $.
\end{thm}

\begin{proof}
Immediate from \eqref{recu} and Lemma~\ref{mink}.
\end{proof}

\begin{prop}\label{p22n-4} The set $\mathscr{H} = \left \{ H_n \right\}_{n \geqslant 10}$ defined in \eqref{accan} is a sequence of recursive graphs. In other words,
\begin{equation}\label{nah}\phi(H_n) = \lambda\phi(H_{n-1}) - \phi(H_{n-2}) \quad \text{for $n \geqslant 12$}.
\end{equation}
\end{prop}

\begin{proof} Along the proof, we make use of the graphs in the family $\mathscr{T}$ defined in \eqref{tienne}. Equation~\eqref{nah} follows from a direct computation when $n \in \{12,13\}$ .
Let now $n \geqslant 14$. Applying \eqref{schwenky1} with respect to the vertex with label $n-7$ in Fig. \ref{Fig1}, we obtain
\begin{align*}
\phi(H_n) &=  \left(\lambda\phi(P_2)^2-2\phi(P_1)\phi(P_2)\right)\phi(T_{n-5})-\phi(P_2)^2\phi(T_{n-6})\\[.2em]
&=\left(\lambda\phi(P_2)^2-2\phi(P_1)\phi(P_2)\right)\left(\lambda\phi(T_{n-6})-\phi(T_{n-7})\right)
-
\phi(P_2)^2\left(\lambda\phi(T_{n-7})-\phi(T_{n-8})\right)\\[.2em]
&=\lambda\left((\lambda\phi(P_2)^2-2\phi(P_1)\phi(P_2))\phi(T_{n-6})-\phi(P_2)^2\phi(T_{n-7})\right)\\
&\phantom{aaaaaaaaaaaaaa}-\left((\lambda\phi(P_2)^2-2\phi(P_1)\phi(P_2))\phi(T_{n-7})-\phi(P_2)^2\phi(T_{n-8})\right)\\[.2em]
&=\lambda\phi(H_{n-1}) - \phi(H_{n-2}).
\end{align*}
Thus, \eqref{nah} is proved.
\end{proof}
In the next proposition, we consider the sequence of polynomials $\mathcal R :=\{g_i\}_{i\geqslant 0}$, where
\begin{align}
g_0  &:= -\lambda^{14} + 9\lambda^{12} -28\lambda^{10} +
35\lambda^8 -13\lambda^6 - 5\lambda^4 + 2\lambda^2 + 1, \nonumber\\
g_1 &:= -\lambda^{13} + 8\lambda^{11} - 21\lambda^9 +
20\lambda^7 - 3\lambda^5 - 4\lambda^3 + \lambda,\nonumber\\
g_2 &:= -\lambda^{12} + 7\lambda^{10} -
15\lambda^8 + 10\lambda^6 + \lambda^4 - \lambda^2 - 1,\nonumber\\
g_3 &:= -\lambda^{11} + 6\lambda^9 -
10\lambda^7 + 4\lambda^5 + 3\lambda^3 - 2\lambda,\nonumber\\
g_4  &:= -\lambda^{10} +
5\lambda^8 - 6\lambda^6 + 2\lambda^4 - \lambda^2 + 1,\nonumber\\
g_5  &:= -\lambda^9 + 4\lambda^7 - 2\lambda^5 - 4\lambda^3 +
3\lambda, \label{ast}\\
g_6  &:= -\lambda^8 + 4\lambda^6 - 6\lambda^4 + 4\lambda^2 - 1,\nonumber\\
g_7  &:= -4\lambda^5 + 8\lambda^3 - 4\lambda,\nonumber\\
g_8  &:= \lambda^8 - 8\lambda^6 + 14\lambda^4 - 8\lambda^2 + 1,\nonumber\\
g_9 &:= \lambda^9 -
8\lambda^7 + 18\lambda^5 - 16\lambda^3 + 5\lambda, \nonumber \\
g_{10}&:= \phi(H_{10})=\lambda^{10}-9\lambda^8+26\lambda^6-30\lambda^4+13\lambda^2-1, \nonumber \\
g_{11}&:= \phi(H_{11})=\lambda^{11}-10\lambda^9+34\lambda^7-48\lambda^5+29\lambda^3-6\lambda, \nonumber \\
g_n  &:= \phi(H_n){\;}\mbox{for}{\;}n \geq 12. \nonumber
\end{align}
\begin{lem}\label{mink2} The polynomials of the sequence $\mathcal R$  satisfy \eqref{recupoly}.
\end{lem}
\begin{proof} For $0 \leqslant n \leqslant 9$ it suffices a direct inspection. For $n \geqslant 10$, the equality $\eqref{recupoly}$ is equivalent to \eqref{nah}.
\end{proof}

\begin{prop}\label{divi2}
For $m \geqslant 1$ and $n \geqslant 10$, $P_m$ divides  $H_n$ if and only if
$$ (m,n) \in \left\{ (1,2k+1) \mid k \geqslant 5 \right\} \cup \left\{ (2,n) \mid n \geqslant 10 \right\} \cup \left\{ (5,6k+5) \mid k \geqslant 1 \right\}.$$
\end{prop}
\begin{proof}
For $m=1$, note that $\phi(P_1)=\lambda$ does not divide $g_0 \in \mathcal R$, whereas $\lambda \mid g_1 \in \mathcal R$. Therefore, by Lemma~\ref{mink}(ii), $\phi(P_1)$ divides
$g_n$ if and only if $n$ is odd. This implies that $P_1$ divides $H_n$ if and only if $n$($\geqslant 10$) is odd.

 Let now $m=2$. The numbers $1$ and $-1$ are roots of $g_0$, $g_1$ and $g_2$. Again from Lemma~\ref{mink}(ii) we infer that $\lambda^2-1$ is a factor of  all the $g_i$'s. This, in particular, means that $P_2$ divides $H_n$ for all $n \geqslant 10$.

From now on we suppose $m \geqslant 3$. Every integer $n \geqslant 3$ can be written in a unique way as $n=(m+1)k+i$ where $k$ is a suitable nonnegative integer and $0 \leqslant i \leqslant m$. Lemma~\ref{mink}(ii) says that $\phi(P_m)$ divides $g_n$ if and only if  $\phi(P_m)$ divides $g_i$.

After a direct inspection, it turns out that, for $ 3 \leqslant m \leqslant 9$ and $0 \leqslant i \leqslant 9$, $\phi(P_m)$  is a factor of $g_i$ only when $(m,i)=(5,5)$; if this is the case, $g_5=-\phi(P_5)(\lambda^4-1)£$. Lemma~\ref{mink}(ii) now ensures that in the range  $ 3 \leqslant m \leqslant 9$, the path $P_m$ divides  $H_n$ if and only if $m=5$ and $n\equiv 5 \bmod 6$.

The proof will end by showing that  $P_m$ with $m\geqslant 10$ never divides a graph in $\mathscr{H}$.  By  Lemma~\ref{mink}(ii) it suffices to show that, for $m\geqslant 10$, none of the $g_i$'s with $i\leq m$ has $\phi(P_m)$ among its factors.

In fact, by looking at the polynomial degrees, for $0 \leq i \leq 4$, only $\phi(P_{14-i})$ could possibly divide $g_i$, but this is not the case: in fact, we can use Lemma~\ref{pathHvD} and \eqref{ast} to compare the coefficients of the monomials of the second largest degree and conclude that $g_i(\lambda)$ is not proportional to $\phi(P_{14-i})$. For $5 \leq i \leq 9$ we have $\deg g_i<10$; hence, $\phi(P_m) \nmid g_i(\lambda)$.\smallskip

Finally, let $m \geqslant i \geqslant 10$. If $\phi(P_m)$ were a factor of $g_i=\phi(H_i)$, the degree of the former could not be larger than $i=\deg g_i$; then, $m$ should be necessarily equal to $i$, implying $\phi(P_m) = \phi(H_i)$ and contradicting Lemma~\ref{pathHvD} which says that every path is DS.
\end{proof}

\begin{lem}\label{T22m} For the graphs in $\mathscr{T}$  and $\mathscr{H}$ respectively defined in \eqref{tienne} and \eqref{accan}, the following inequalities hold.\smallskip

\noindent $\mathrm{(i)}$ $\rho(T_h) < \rho(T_k)$  for  $6\leqslant h <  k$.\smallskip

\noindent $\mathrm{(ii)}$  $\rho(H_s) > \rho (H_t)$ for $10 \leqslant s <t$.
\end{lem}
\begin{proof} Part (i) follows from Lemma \ref{proper} since $T_h \subset T_k$ whenever $h<k$. Part (ii) is a consequence of Lemma \ref{internalpath}(ii) since $H_t$ is obtained from $H_s$ by subdividing the internal path $t-s$ times.
\end{proof}

\begin{lem}\label{third}
For $n \geqslant 10$, $\lambda_3(H_n) < 2$.
\end{lem}

\begin{proof}
A direct computation shows that $\lambda_3(H_{10})=1<2$. Let now $n>10$. Let $w'$ be the vertex of $H_n$ whose label is $2$ in  Fig.~\ref{Fig1} an let $w''$ the only vertex of degree $3$ in $T_{n-5}$. From Lemmas~\ref{path} and~\ref{interlace} it follows that
$$\lambda_3(H_n) \leqslant \lambda_2(H_n -
w') = \lambda_2(2P_2 \cup T_{n-5}) \leq \lambda_1(2P_2 \cup
(T_{n-5}-w'')) = \lambda_1(4P_2 \cup P_{n-10}) < 2,$$ as claimed.
\end{proof}

\begin{prop}\label{divi3}
For $n=2k+1>10$, the T-shape graph $T_h$ divides the H-shape graph $H_n$ if and only if $h=k\geqslant 6$.
\end{prop}
\begin{proof} Let $k\geqslant 6$. Applying \eqref{schwenky2} to $H_{n}$ with respect to the edge $\tilde{e}=\tilde{u}\tilde{v}$ where $\tilde{u}$ (resp., $\tilde{v}$)  is the vertex labelled $k-3$ (resp., $k-2$) on the horizontal path of $H_n$ in Fig.~\ref{Fig2}, we obtain
\begin{equation}\label{DKfact} \phi(H_{2k+1})= \begin{cases} \phi(T_6) \,\left( \phi (T_7) - \phi(P_5) \right) \qquad  \;\;\; \text{for $k=6$};\\[.3em]
  \phi(T_k)\left( \phi (T_{k+1}) - \phi (T_{k-1})\right) \quad \text{for $k>6$}.\\
\end{cases}
\end{equation}
This proves the `if' part. Suppose now
\begin{equation}\label{mink3} \phi(H_{2k+1})= \phi(T_h) f(\lambda)
\end{equation} for a suitable polynomial $f(\lambda)$. As a consequence of \cite[Theorem 1]{wang1}, $f(\lambda)$ has a positive degree; therefore,  it makes sense to refer to its $i$-th largest root $\lambda_i(f)$ for $1 \leqslant i \leqslant \deg f(\lambda)$.
We explicitly note that, from the first part of the proof, a situation of type \eqref{mink3} actually occurs at least when $h=k$. From Lemma \ref{bro}, we see that
$\lambda_1(T_h) < {\mathfrak h} < \lambda_1(H_{2k+1})$; therefore,   $\lambda_1(H_{2k+1})=\lambda_1(f)$. Clearly, $\lambda_2(H_{2k+1}) \in \{ \lambda_2(f), \lambda_1(T_h) \}$. The case $\lambda_2(H_{2k+1}) =\lambda_2(f)$ cannot occur, otherwise $$\lambda_3(H_{2k+1}) = \max \left\{\lambda_1(T_h), \lambda_3(f) \right\} \geqslant \lambda_1(T_h) \geqslant
2,$$  which contradicts Lemma \ref{third}.
Thus, $\lambda_2(H_{2k+1}) =\lambda_1(T_h)$. The same argument, starting from \eqref{DKfact} instead of \eqref{mink3} allows us to conclude that $\lambda_2(H_{2k+1})$ is also equal  to $\lambda_1(T_k)$; hence,  $h$ is necessarily equal to $k$  by Lemma
\ref{T22m}.
\end{proof}

\begin{lem}\label{mink4}
For each $n\geqslant 10$, the polynomial $\phi(H_n,\lambda)$ evaluated at $2$ gives $\phi(H_n,2)=9(n-15)$.
\end{lem}
\begin{proof} We can  either apply \cite[Lemma 4]{gha-omi-tay} or make use of an inductive argument based on  the evaluations $\phi(H_{10},2)=-45$ and $\phi(H_{11},2)=-36$ made available by\eqref{ast} and the identity \eqref{nah}.
\end{proof}

\begin{prop}\label{divi4}

For $n \geqslant 10$ and $m  \geqslant 3$, $C_m$ divides $
H_n$ if and only if $(m,n) \in \left\{ (3,15), (6,15) \right\}$.
\end{prop}
\begin{proof} From Lemma~\ref{path} or, equivalently, from  Lemma \ref{smith1}, one sees that $(\lambda -2) \mid \phi(C_m)$ for every $m \geqslant 3$, whereas, as a consequence of Lemma~\ref{mink4}, $2 \in {\rm sp}(H_n)$ only if $n=15$. Thus, if $C_m$ divides $H_n$, then $m \leqslant n=15$. By a direct calculation,
$$\phi(H_{15}) =
\lambda(\lambda^2-4)(\lambda^2-1)^3h(\lambda), \qquad \text{where $h(\lambda):= \lambda^6-7\lambda^4+12\lambda^2-2$}.$$
Now, it is clear that $\phi(C_3)= (\lambda-2)(\lambda+1)^2$ and $\phi(C_6)= (\lambda^2-4)(\lambda^2-1)^2$  divide $\phi(H_{15})$. Moreover, $C_m \nmid H_{15}$ for $m \not \in \{3,6\}$, since $\phi(C_4)=\lambda^2(\lambda^2-4)$ and, for $m\not\in\{3,4,6\}$,
$\lambda_2(C_m)=2\cos(2 \pi /m)$ does not belong to ${\rm sp}(H_{15})$. In order to see this, note that
$$0=2 \cos \left( \frac{2\pi}{4} \right)  < \zeta_1 < 2 \cos \left( \frac{2\pi}{5} \right) \qquad  \text{and} \qquad  2 \cos \left( \frac{2\pi}{9} \right) < \zeta_2 < 2 \cos \left( \frac{2\pi}{10}\right),$$
where $\zeta_1 \approx 0.4317$ and $\zeta_2 \approx 1.5718$ are the two positive roots of $h(\lambda)$.
 \end{proof}

In view of the last results in this section concerning the divisibility of the $H_n$'s, we state here two propositions involving  $\rho_n := \rho(H_n)$ ($n \geqslant 10$).
\begin{prop}\label{index} The index $\rho_n$ of $H_n$ belongs to the interval
$\left({\mathfrak h},  \frac{3\sqrt{2}}{2} \right)$
 if and only if $n \geqslant 12$.
\end{prop}
\begin{proof} The inequality $\rho_n >{\mathfrak h}$ holds for every $n \geqslant 10$ because of the following two facts:\\
$\bullet$ by Lemmas~\ref{smith}-\ref{bro}, $H_n \not\in {\mathscr G}_{<2} \cup {\mathscr G}_{2} \cup {\mathscr G}_{(2,{\mathfrak h})};$\\
$\bullet$ the number $\mathfrak h$ is not a graph eigenvalue, since its minimal polynomial has nonreal roots.

By a direct computation,
$$
\rho_{10} > 2.17008 >
\rho_{11}=\frac{\sqrt{10+2\sqrt{17}}}{2} >\frac{3}{\sqrt{2}} \quad \text{and} \quad
\rho_{12}< 2.115 < \frac{3}{\sqrt{2}}.$$
Now, by Lemma \ref{T22m}(ii) we have
$
\rho_n<
\rho_{12}< 3/{\sqrt{2}}$ for all $n > 12$.
\end{proof}

\begin{prop}\label{C122S}
Let $n \geqslant 12$. If $\rho(C_{1,2; 2s}^{0,s}) \leqslant
\rho_n$, then $s \geqslant 7$.

\end{prop}

\begin{proof}
For each $s\geqslant 2$, the graph $C_{1,2; 2(s+1)}^{0,s+1}$ is obtained from  $C_{1,2; 2s}^{0,s}$ by subdividing two edges on the cycle incident to the same vertex of degree $3$;
therefore, $\rho(C_{1,2; 2s}^{0,s})>\rho(C_{1,2; 2(s+1)}^{0,s+1})$ by Lemma \ref{internalpath}(ii). Thus,
$$\rho(C_{1,2;4}^{0,2}) > \rho(C_{1,2;6}^{0,3}) > \rho(C_{1,2;8}^{0,4}) > \rho(C_{1,2;10}^{0,5}) >
\rho(C_{1,2;12}^{0,6}) \approx 2.1215 > 3/\sqrt{2} >
\rho_n,$$ where the last inequality comes from Proposition~\ref{index}.
\end{proof}

\begin{lem}\label{C0k}{\rm \cite[Theorem 4.1]{wjf}} $\rho(C_{s,t;g}^{0,k}) < \rho(C_{s,t;g}^{0,k-1})$ whenever
$\displaystyle 2 \leqslant k \leqslant \left\lfloor \frac{g}{2} \right\rfloor$ and $(s,t) \in \mathbb N \times \mathbb N$.
\end{lem}

\begin{prop}\label{divi5} Let $n \geqslant 12$, No triples $(t,s,k)$ exist with
$t > 1$, $s > 3$ and $ 1< k \leqslant s$ such that
$\phi(C_{1,t,2s}^{0,k})= (\lambda^2-1)^{-1} \phi(H_n)$.
\end{prop}
\begin{proof} Assume by contradiction that there exists a closed quipu $CQ:=C_{1,t;2s}^{0,k}$ with $t>1$, $s>3$ and $1<k\leqslant s$  such that
$\phi(H_n) = \phi(P_2) \phi(CQ)$. Clearly,
$
\rho_n= \rho(CQ)$. By  Lemma~\ref{C0k} and interlacing we obtain
$$
\rho_n =  \rho(CQ) \geqslant \rho(C_{1,t;2s}^{0, s}) \geqslant \rho(C_{1,2;2s}^{0, s}),$$  implying from Proposition \ref{C122S} that $s \geqslant 7$ and consequently, the girth $2s$ of $CQ$ is at least $14$.  Thus, $P_{2;15}^8 \subseteq P_{2;2s+1}^{s+1}$ which, together with $P_{2;2s+1}^{s+1} \subset C_{1,2;2s}^{0, s}$ and the above inequality, yields
\begin{equation}\label{P1214}
\rho_n> \rho(P_{2;2s+1}^{s+1}) \geq \rho(P_{2;15}^8) = 2.0904+.
\end{equation}
Note that, for $n\geqslant 15$, $
\rho_n \leqslant
\rho_{15} \approx 2.08397$; therefore, \eqref{P1214}
is only possible for $n \in \{12,13,14 \}$. Consequently, $CQ$ should have at most  $12$ vertices, contradicting the lower bound detected above for its girth.
\end{proof}

\section{Eigenvalues location for  the tree $H_n$ and some related closed quipus}\label{S4}
In Section~\ref{S3} we already achieved some results related to the location of specific eigenvalues of $H_n$  (see Lemma~\ref{third} and Proposition~\ref{index}).
We list below the approximated values of $\rho_n:=\rho(H_n)$ up to $n=20$.
Coherently with Proposition~\ref{index}, for $n \geqslant 12$ the values in \eqref{kekkaz} belong to the interval $\left({\mathfrak h}, 3/\sqrt{2}\right)$.
\begin{align}
 \nonumber \rho_{10} \approx 2.17009  &&\nonumber  \nonumber \rho_{13} \approx 2.10100 &&\rho_{16} \approx 2.07852 && \rho_{19} \approx 2.06843 \\
\label{kekkaz}  \rho_{11}  \approx 2.13578 && \rho_{14} \approx 2.09118  &&  \rho_{17} \approx 2.07431 &&  \rho_{20} \approx 2.06637. \\
\nonumber  \rho_{12} \approx 2.11491 && \rho_{15} \approx  2.08397  && \rho_{18} \approx  2.07103
\end{align}
The following proposition concerns the location of $\lambda_2(H_n)$. Thereafter, we discuss the relation between the spectral radii of $H_n$ and  some closed quipus.

\begin{prop}\label{second}
Let $\lambda_2(H_n)$ be the second largest eigenvalue of $H_n$. Then,

$$\lambda_2(H_n) \;
\begin{cases}
< 2 &\mbox{for $ 10 \leqslant n < 15$}; \\
= 2 &\mbox{for  $n = 15$};\\
\in \left(2, {\mathfrak h}\right) &\mbox{for $n > 15$}.
\end{cases}
$$

\end{prop}
\begin{proof} For convenience, let $\lambda_2 := \lambda_2(H_n)$. 
From standard calculus,
$$\phi (H_n, c) \begin{cases} \, >0  & \text{if $c \in (\lambda_3(H_n), \lambda_2)$},\\
  <0 & \text{if $c \in (\lambda_2, \rho_n)$}.
\end{cases}
$$
Since  $\rho_n > {\mathfrak h} >2$ (by Proposition~\ref{index}) and
$\lambda_3(H_n)<2$ (by Lemma~\ref{third}), it suffices Lemma~\ref{mink4} to see that $\lambda_2<2$ or $\lambda_2>2$ according if $n<15$ or $n>15$, and $\lambda_2=2$ for $n=15$.

The upper bound for $\lambda_2$ when $n>15$ is proved as follows: by using interlacing and
Lemma~\ref{bro},  $$\lambda_2(H_n) \leqslant \rho(H_n - w) = \rho(2P_2 \cup P_{2,n-7}^2) < {\mathfrak h},$$ where $w$ is a vertex of degree 3 in $H_n$.
\end{proof}
\begin{prop}\label{T8}
Among the trees of order $n$  of type ${\mathcal T}^{10}$ (see Fig.~\ref{Fig3}), the smallest spectral radius is only achieved by $H_n$.
\end{prop}
\begin{proof} Let $G$ be a graph with $n$ vertices in ${\mathcal T}_{10} \setminus \{H_n\}$: Since $G$ is a tree, then $\varepsilon_G= \varepsilon_{H_n}=n-1$. Recall that the length of each dotted line in Fig.~\ref{Fig3} is at least 2.

Let $G'$ be the graph obtained from $G$ by leaving exactly two edges on each of its four pendant paths, and let $G''$ be the the graph obtained from $G'$ by inserting the deleted edges in the internal path of $G'$. Clearly $G''=H_n$, and
$\rho(G) >\rho(G') >
\rho_n$. The former inequality comes from Lemma~\ref{proper}, the latter from
Lemma~\ref{internalpath}.
\end{proof}
  \begin{prop}\label{C1g0} Among all the closed quipus of order $\nu$, the lollipop $C_{1,\nu-1}^0$
has the smallest spectral radius. Moreover, ${\mathfrak h}  <
\rho(C_{1,\nu-1}^0) \leqslant \rho_{10}$, where $\rho(C_{1,\nu-1}^0) = \rho_{10}$ (resp., $\rho_{11}$) if and only if $\nu=4$ (resp., $\nu=5$).
\end{prop}
\begin{proof}
For the first conclusion, one can use Lemmas~\ref{proper} and~\ref{internalpath}, precisely as done in the proof of Proposition~\ref{T8}; therefore, there is no need to write down the details.

The quickest way to deduce ${\mathfrak h}  < \rho(C_{1,\nu-1}^0)$ is  to look at Lemmas~\ref{smith}, \ref{smith1}, and~\ref{bro}, checking that the set ${\mathscr G}_{<2} \cup {\mathscr G}_{2} \cup {\mathscr G}_{(2,{\mathfrak h})}$ does not contain closed quipus (recall that no graph has $\mathfrak h$ is its spectrum).

For the remaining part of the statement, a direct check shows that $C_{1,3}^0 \mid H_{10}$ and $C_{1,4}^0 \mid H_{11}$; furthermore, $\rho(C_{1,3}^0)=\rho_{10}$ and $\rho(C_{1,4}^0)=\rho_{11}$. Thus, Lemma~\ref{internalpath} implies that
\[\rho(C_{1,\nu-1}^0) < \rho(C_{1,4}^0)= \rho_{11} < \rho(C_{1,3}^0) = \rho_{10} \qquad \text{for all $\nu>5$}. \qedhere \]
\end{proof}

\begin{prop}\label{CP}
Let $\mathcal C \mathcal Q_{2s}$ the set of all closed quipus with girth $2s$. If there exists a graph $G \in \mathcal C \mathcal Q_{2s}$ such that $\rho(G)=
\rho_n$ for $n \geqslant 12$ and $n \neq 13$, then, $s\geqslant 4$.
\end{prop}
\begin{proof} We show that no graphs in $\mathcal C \mathcal Q_{4} \cup \mathcal C \mathcal Q_{6}$ satisfy the condition of the statement.
Since every $G \in \mathcal C \mathcal Q_{4}$ contains $C_{1;4}^0$ among its subgraphs, by interlacing and Lemma~\ref{internalpath} we obtain
\[ \rho(G) \geqslant \rho(C_{1;4}^0) = \rho_{11} >
\rho_n \quad \text{for $n \geqslant 12$}.\]
A direct check shows that $C_{1;6}^0 \mid H_{13}$ and $\rho(C_{1;6}^0)=
\rho_{13}$ (the latter result also comes from the subsequent Proposition~\ref{C01P}).
Therefore, for every $G \in \mathcal C \mathcal Q_{6} \setminus \{  C_{1;6}^0 \}$, we have
\[ \rho(G) > \rho(C_{1;6}^0)=
\rho_{13} >
\rho_n \quad \text{for $n > 13$}.\]
Moreover, among the subgraphs of $G$ we find at least one graph in the set
\[\left\{C_{2;6}^0, \, C_{1,1;6}^{0,1}, \, C_{1,1;6}^{0,2}, \,C_{1,1;6}^{0,3}\right\}.\] Hence,
$ \rho(G) \geqslant \min \left\{\rho(C_{2;6}^0), \, \rho(C_{1,1;6}^{0,1}), \, \rho(C_{1,1;6}^{0,2}),\, \rho(C_{1,1;6}^{0,3})\right\} >2.13578 >
\rho_{12} \approx 2.11491.$
\end{proof}

\begin{lem}\label{P1-P2}{\rm \cite[Lemma 4.2]{wjf}}
Let $(r_1,r_2)$ and $(s_1,s_2)$ two pairs of nonnegative integers such that $r_1 \leqslant r_2$, $r_1 < s_1 \leqslant s_2$ and $r_1+r_2 = s_1+s_2$.
Then,
$$\phi(P_{r_1})\phi(P_{r_2}) - \phi(P_{s_1})\phi(P_{s_2}) =
-\phi(P_{s_1-r_1-1})\phi(P_{s_2-r_1-1}).$$

\end{lem}
From Lemma~\ref{P1-P2} it follows in particular
that \begin{equation}\label{pippi} \phi(P_{k-5})\phi(P_{k-2}) = \phi(P_{k-4})\phi(P_{k-3})  -\lambda
\quad \mbox{and} \quad \phi(P_{k-6})\phi(P_{k-4}) = \phi^2(P_{k-5})
 -1.
\end{equation}
for all $k\geqslant 5$.

Along the proof of Proposition~\ref{CP} we noted that  $\rho(C_{1;4}^{0})=\rho_{11}$  and  $\rho(C_{1;6}^{0})=\rho_{13}$. Next Proposition shows that these are peculiar cases of a more general phenomenon.
\begin{prop}\label{C01P}
For $s \geqslant 2$ and $k \geqslant 5$,  $\rho(C_{1;2s}^{0})=\rho_{2k+1}$ if and only if $s = k-3$.
\end{prop}

\begin{proof}
Let $s=k-3$.
By suitably applying \eqref{schwenky2}, \eqref{path-poly} and  \eqref{pippi},
we obtain the following sequence of equalities:
\begin{align*}
\phi(C_{1;2k-6}^0) &= \lambda\phi(C_{2k-6}) - \phi(P_{2k-7})\\
&= \left(\lambda\phi(P_{2k-6}) - \phi(P_{2k-7})\right) - \lambda\phi(P_{2k-8}) - 2\lambda\\
&= \phi(P_{2k-5}) - \lambda\phi(P_{2k-8}) - 2\lambda \label{waffa} \\
&= \left(\phi(P_{k-4})\phi(P_{k-1}) - \phi(P_{k-5})\phi(P_{k-2})\right) -
   \lambda \left(\phi^2(P_{k-4}) - \phi^2(P_{k-5})\right) -2\lambda\\
&= \phi(P_{k-4})\phi(P_{k-1}) - \lambda\phi^2(P_{k-4}) -
   \phi(P_{k-4})\phi(P_{k-3}) + \lambda\phi(P_{k-6})\phi(P_{k-4})\\
&= \phi(P_{k-4})q(\lambda).
\end{align*}
where $q(\lambda)= \phi(P_{k-1}) - \lambda\phi(P_{k-4}) -
   \phi(P_{k-3}) + \lambda\phi(P_{k-6})$.

Through \eqref{schwenky2}, we similarly decompose $\phi(H_{2k+1})$:
\begin{align*}
\phi(H_{2k+1}) &= \phi(T_k)\left(\phi(T_{k+1})
    - \phi(T_{k-1})\right)\\[.2em]
&= \phi(T_k)\left(\phi(P_2)\phi(P_{k-1}) - \lambda\phi(P_2)\phi(P_{k-4})
    - \phi(P_2)\phi(P_{k-3}) + \lambda\phi(P_2)\phi(P_{k-6})\right)\\[.2em]
&= \phi(P_2)\phi(T_k)q(\lambda).
\end{align*}
So far, we have proved that $\phi (C_{1;2k-6}^0)=  \phi(P_{k-4})q(\lambda)$ and $\phi(H_{2k+1})=   \phi(P_2)\phi(T_k)q(\lambda)$.
We now claim that
$$ \rho (C_{1;2k-6}^0) = \rho_{2k+1}=q_1,$$
where $q_1$ is the largest root of $q(\lambda)$. In fact,
$$ \rho (C_{1;2k-6}^0) > \rho(\phi(P_{k-4})) \qquad \text{and} \qquad
\rho_{2k+1} >\rho(T_k)>\rho(P_2),$$
 since
$P_{k-4} \subset C_{1;2k-6}^0$ and $P_2 \subset T_k \subset H_{2k+1}$. Thus,
the `if' part of the statement is proved. \smallskip

The `only if' part follows easily from the `if' part and Lemma~\ref{T22m}(ii). In fact,
$$\rho(H_{2k'+1}) < \rho(C_{1;2s}^{0}) = \rho(H_{2s+7}) < \rho(H_{2k''+1}) $$
for $k'>s+3$ and $5 \leqslant k'' < s+3$.
\end{proof}

\section{Proof of Theorem \ref{main}}\label{S6}
We start by recalling a very useful lemma by Liu and Huang.
 \begin{lem}\label{LHuang} {\rm \cite[Equation (2) and Lemma 3.1]{liu_huang}}
Let $G$ be a graph cospectral to an H-shape graph with $n$ vertices. Then, the number $n_4(G)$ of quadrangles in $G$ is at most $1$. Moreover, the degree sequence of $G$ is $(1^2, 2^{n-2})$ if $n_4(G)=1$, whereas it
  belongs to $\{ (0^1, 1^1, 2^{n-3}, 3^1), (1^2, 2^{n-6},3^2) \}$ if $n_4(G)=0$.
\end{lem}
The trees $H_{10}, H_{13}$ and $H_{15}$ are not DS by \eqref{035}.
Thus, the graph of minimal order we need to consider is $H_{11}$. For this specific H-shape tree we state Proposition~\ref{main1}, whose proof is posponed
to Appendix B.
\begin{prop}\label{main1}
The tree $H_{11}$ is  DS.
\end{prop}
In view of  \eqref{035} and Proposition \ref{main1}, from now on we can  assume $n \geqslant 12$ and $n \not\in \{13,15\}$. Our goal is to prove that, for $n$ in that range, $H_n$ is DS.

Let  $G$ be a graph such that $\phi(G) = \phi(H_n)$. Due to  Proposition~\ref{invariant}, $G$ is surely bipartite with $\nu_G = \varepsilon(G) + 1 = n (\geqslant 12)$. By Proposition~\ref{index} and~\ref{second} we infer that $G$  contains exactly one component out of $\mathscr{G} :=\mathscr{G}_{<2} \cup \mathscr{G}_2 \cup \mathscr{G}_{(2,{\mathfrak h})}$. By Lemmas~\ref{smith}--\ref{woo}, every component of $G$ is a tree or a unicyclic graph. This information, together with $\nu_G = \varepsilon_G + 1$, allows us to conclude that one component of $G$, say $T$, is tree, an the remaining ones (if any) are unicyclic.  In addition, no component of $G$ is a cycle, otherwise $\phi(C_k)$ should divide $\phi(G)$ for a suitable $k$, and this never occurs for $n \neq 15$ (see Proposition~\ref{divi4}).

By collecting all this information together, and recalling that $\mathscr{G}$ only contains trees and cycles, we see that there are just two possibilities: either $G=T$ or $G= T \cup CQ$, where $CQ$ is a closed quipu with $\rho(CQ)=\rho_n$. \medskip

{\it Case 1.} $G =T$.   From \cite[Lemma 3.1]{liu_huang}, surely $\Delta_T<4$. Thus, $G=T$ is an open quipu by Lemma~\ref{woo}.
Due to $\nu(T) = n$ and $M_2(T) = M_2(H_{11})$, by the first two parts of Lemma \ref{2match}(ii), $\Delta_T=3$ and $k_T=2$. In other words,  $T$ is an open quipu of type $P_{n_1,n_2,a}^{m_1,m_2}$. By Proposition~\ref{ahah}(ii), $M_3(T) = M_3(H_{11})$; therefore, $T$ belongs to ${\mathcal T}_{10} \cup {\mathcal T}_{11}$ by  Corollary~\ref{3match1}.\medskip

{\it Case 1.1.} $T \in \mathcal{T}_8$.  Proposition~\ref{T8} ensures that $T$ is necessarily equal to $H_{n}$. We shall see that neither of the subsequent cases can actually occur.\medskip

{\it Case 1.2:} $T \in \mathcal{T}_{11}$. In this case, $T=P_{n_{1},n_{2},a}^{1,2}$ for some $n_{1},n_{2} \geqslant 2$ and $a\geqslant 5$. Since $\nu_{T}\geqslant 12$, then $T'':=P_{2,2,5}^{1,2}$ is a proper subgraph of $T$. By using interlacing with respect a vertex of $T$ with degree $3$ , we realize that $\lambda_{2}(H_{11})=\lambda_{2}(T)<2$. Therefore $n \in \{12,14\}$ by Proposition~\ref{second}. After computing $\phi(T'')$ and $\phi(H_{11})$, we see that the spectral radius of both polynomials is the largest root of $\lambda^4-5\lambda^2+2$. Hence, by Lemmas~\ref{proper} and ~\ref{T22m}(ii), we arrive at
\[  \rho_{14} <\rho_{12} < \rho_{11}  = \rho (T'') <\rho(T) \in \{\rho_{12} , \rho_{14} \}, \]
which is contradictory. \medskip

{\it Case 2:} $G=T \cup CQ$ and $\rho(G) = \rho(CQ)$. We set $\nu_1=\nu_T$ and $\nu_2=\nu_{CQ}$. Clearly, $\nu_1+\nu_2=n$. Since $G$ is bipartite, $CQ$ has an even girth, and such girth $g$ is at least $8$ by  Proposition~\ref{CP}.
From Lemma~\ref{LHuang} we deduce that the degree sequence of $G$ is  either $(3^1, 2^{n-3},1^1,0^1)$ or   $(3^2, 2^{n-6},1^4)$.
Before analizying the two possible occurrences, we note that, in any case, $(M_2(G),M_3(G))=(M_2(H_n),M_3(H_n))$ by Proposition~\ref{ahah}(ii). As in Section~\ref{S2}, we set $f(n)= (n^3 - 12n^2 + 35)/6$. \medskip

{\it Case 2.1:} The degree sequence of $G$ is $(3^1, 2^{n-3},1^1,0^1)$. In this case   $G=P_1 \cup C^0_{\ell;n-\ell-1}$. With the aid of Propositions~\ref{Match}(ii), we compute
\[ M_3(G)=\begin{cases} f(n)+1 & \text{if $\ell=1,$}\\[0.2em]
									      f(n)+2 & \text{if $\ell\geqslant 2$.}
\end{cases}
\]
In any case, by Corollary~\ref{3match1},  $M_3(G) < M_3(H_n)=f(n)+4$, against the assumption of cospectrality on $G$ and $H_n$.\medskip

{\it Case 2.2:} The degree sequence of $G$ is $(3^2, 2^{n-6},1^4)$. In this case, the number $k_T$ of vertices of degree $3$ in $T$ is at most $1$. Unfortunately, no restriction on $T$ and $CQ$ comes from $M_2(G)=M_2(H_n)$; therefore we need to compare $M_3(G)$ and $M_3(H_n)$.

\medskip

{\it Case 2.2.1:} $k_T= 0$. The restrictions imposed by the degree sequence and Proposition~\ref{divi2} lead to
\begin{equation*}
G = P_{\nu_1} \cup CQ ,  \qquad \text{where} \quad CQ=C_{n_{1},n_{2};g}^{0, m}, \; \nu_1 \in \{2,5\},   \quad \text{and} \quad  g \geqslant 8.
\end{equation*}
 Let $e_{_{i,j}}$ denote the number of edges of type $uv$ with $d(u) = i$ and $d(v) = j$ in the graph $CQ$.
 Setting $e_{_{3,3}} = x$ and $e_{_{1,2}}=y$, it is clear that  $x \in \{0,1\}$ and $y \in \{0,1,2\}$.
It is not hard to prove that
$$e_{_{1,3}} = 2-y, \qquad e_{_{2,3}} = 4-2x+y \qquad \text{and} \qquad e_{_{2,2}} = \nu_2+x-y-6.$$
A counting argument, together with Parts~(ii) and~(iii) of Lemma~\ref{Match}, leads to
\begin{equation}\label{3Cmathch}
M_3(G) = \left\{\begin{array}{cc}
M_3(CQ) + M_2(CQ) = f(n)+x+y+3\phantom{nullnu} &\mbox{if $\nu_1 = 2$;}\\[1mm]
M_3(CQ) +
4M_2(CQ) + 3\nu_2 = f(n)+x+y+2 &\mbox{if $\nu_1 = 5$}.
\end{array}
\right.\end{equation}

{\it Case 2.2.1.1:}  $\nu_1 = 2$.  The equality $M_3(G)=M_3(H_n)=f(n)+4$ only holds for $(x,y) \in \left\{(1,0), (0,1) \right\}$.
If $(x,y) = (1,0)$, then $CQ = C_{1,1;n-4}^{0,1}$. In this case, $P_{1,1;8}^{3,4} \subset C_{1,1,g}^{0,1}$ and, by Lemma \ref{proper},
$$\rho_n= \rho(CQ) =\rho(C_{1,1,n-4}^{0,1}) > \rho(P_{1,1,8}^{3,4}) = 2.1267+ >\rho_{12} >\rho_n,$$ a contradiction.\medskip

If $(x,y)= (0,1)$, then $CQ \cong C_{1,t;2s}^{0,k}$, where $t >1$, $s\geqslant 4$ and $1 < k \leqslant s$. Yet, the equality
\[ \phi(H_n) = \phi(G) = \phi(P_2)\phi(C_{1,t,2s}^{0,k}) \] contradicts Proposition~\ref{divi5}.\medskip

{\it Case 2.2.1.2:} $\nu_1 = 5$. This time, the equality $M_3(G)=M_3(H_n)$ and \eqref{3Cmathch} both hold only if $(x,y) \in \left\{ (0,2),(1,1) \right\}$.

If $(x,y) = (0,2)$, then $CQ \cong C_{t_1,t_2,2s}^{0,k}$, where $1 < k \leqslant s$ and $t_1,t_2 > 1$.
Thus, the graph $CQ$ contains as a proper subgraph
$$K':= C_{2,2;8}^{0,4} \qquad \text{for $(k,s)=(4,8)$} \qquad \text{or} \qquad K'':= P_{2,2,k+5}^{2,k} \quad \text{otherwise}.
$$
Now, it turns out that $\rho(K') >2.17532 >\rho_n$; moreover, from $\nu_{K''} \leqslant n-5$, Lemma \ref{proper}  and  Lemma \ref{internalpath}(ii),
\[ \rho(K'') > \rho (H_{n-7}) >\rho(H_n).\]
In any case $\rho(G)>\rho_n$ against the initial assumption.

Suppose, now, $(x,y) = (1,1)$. In this case, then $CQ \cong C_{1,t;2s}^{0,1}$, where $s \geqslant 4$ and $t >1$. We note that $K''':=P_{1,2;6}^{2,3} \subset
CQ$. Since  $\phi(K''')= \lambda(\lambda^2-1)(\lambda^2-2)(\lambda^4-5\lambda^2+2)$, we find $\rho(K''')=\rho_{11}$ and, by Lemmas~\ref{proper} and~\ref{T22m}(ii),
$$\rho_n=\rho(G) = \rho(CQ) > \rho(K''') = \rho_{11}>\rho_{12} \geqslant \rho_n,$$ which  is impossible.\medskip

{\it Case 2.2.2:} $k_T= 1$. From $T \in \mathscr{G}$ it follows that $T$ is a T-shape tree of type
$P_{a;b+c+1}^{b}  $ where $a \leqslant b \leqslant c$ and $\rho(T) < \mathfrak h$. Thus,
\begin{equation*}
G = P_{a;b+c+1}^{b} \cup C_{\ell;
2s}^{0}.
\end{equation*}
We denote by $e_{_{i,j}}(T)$ (resp., $e_{_{i,j}}(CQ)$) the number of edges in $T$ (resp., $CQ$) of type $uv$ with
$d(u) = i$ and $d(v) = j$. Once we set $e_{1,2}(T) = x$ and  $e_{1,2}(CQ) = y$, we see that
\[\begin{array}{lllll}
e_{_{1,3}}(T) = 3-x, &&&       e_{_{1,3}}(CQ) = 1-y,              \\
e_{_{2,3}}(T) = x,    &&&  e_{_{2,3}}(CQ) = 2+y,\\
e_{_{2,2}}(T) = \nu_1-x-3, &&& e_{_{2,2}}(CQ) = \nu_2-y-3.
\end{array}
\]
Evidently,  $0 \leqslant x \leqslant 3$ and $0 \leqslant y \leqslant 1$. From Lemma~\ref{Match}, it follows that
\begin{align*}\label{3CCmathch}
M_3(G) &= M_3(T) +
M_3(CQ) + \nu_2M_2(T)+(\nu_1-1)M_2(CQ)\\
&= \frac{(\nu_1 + \nu_2)^3 - 12(\nu_1 + \nu_2)^2 + 35(\nu_1 + \nu_2)
+ 6(x+y)}{6}\\
&= f(n) + x+y.
\end{align*}
Since $M_3(H_n)=f(n)+4$, we arrive at $x+y=4$, which only occurs when  $(x,y)=(3,1)$. The subset of $\mathscr G$ of T-shape graphs with no pendant vertices adiacent to the vertex of degree $3$ is $\mathscr{P}:=\{P_{2;7}^3,\,
P_{2;c+3}^2 \mid c>1 \}$. Thus, $T \in \mathscr{P}$ and $CQ=C_{\ell,2s}^0$ with $\ell > 1$, since $y=1$.\smallskip

Now, surely $T\not=P_{2;7}^3$; in fact, $\phi(P_{2,7}^3) = \phi(P_3)(x^6-6x^4+8x^2-2)$, but $P_3$ does not divide $H_n$
by Proposition~\ref{divi2}. The only other possibility is
\begin{equation}\label{G4}
G = P_{2,c+3}^2 \cup C_{\ell, 2s}^{0}  \qquad \text{with $c>1$, $s \geqslant 4$ and $\ell>1$.}
\end{equation}

{\it Case 2.2.2.1:} the integer $n$ is even. From Proposition~\ref{DK-low} it follows that $\vartheta(G) =
\vartheta(H_n) \in \{ \pm 1\}$. Therefore, if \eqref{G4} holds, the lowest term of $T=P_{2,c+3}^2$ must be a nonzero constant. By Proposition~\ref{T22c-low} we deduce that $\vartheta(T) \in \{ \pm 1\}$ and $\nu_1$ is even. Now, by $n = \nu_1 + 2s + \ell$ we obtain that $\ell$ is also even; hence, by Proposition \ref{lollipop-low}, $\vartheta(CQ)$ should be
in $\{\pm 4\}$, whereas, from \eqref{G4}, $\vartheta(CQ)= \vartheta(G)/\vartheta(T)$ should be in $\{ \pm  1\}$, proving that this case cannot occur.\medskip

{\it Case 2.2.2.2:} the integer $n$ is odd. Let $n = 2k+1$. Since $\nu_1\geqslant 5$ and $\nu_2\geqslant 10$, and $n\neq 15$, then $k \geqslant 8$.   From $\phi(T) \mid \phi(H_n)$ and Proposition \ref{divi3} we obtain $c = k-5$ or, equivalently, $T = P_{2,k-2}^2$. It follows that $\nu_2= k+1$ and, consequently, $2s < k$.
In our hypotheses, $k \geq 8$ and $\ell>1$. Then, $2s<k<2k-6$ and $ C_{1, 2s}^{0} \subset C_{\ell,2s}^{0}$. Together with Lemma~\ref{internalpath}(ii) and Proposition~\ref{C01P} we arrive at

$$\rho(H_{2k+1})=\rho(C_{\ell, 2s}^{0}) > \rho(C_{1,2s}^{0}) > \rho(C_{1,2k-6}^{0}) = \rho(H_{2k+1}),$$ which is false.

The end of this long case analysis also marks the end of the proof. In fact, we have proved that if $G$ and $H_n$ are cospectral with $n \not\in \{10, 13, 15 \}$, then $G=H_n$.

\section{Conclusions}
In this paper we achieved several results concerning the divisibility of recursive graphs, focussing on the family of  H-shape trees  $H_n$  ($n \geqslant 10$) and studying which  of them is determined by its adjacency spectrum.
The investigation of divisibility in the realm of simple graphs could be useful to shed some light on a quite intriguing general problem, namely: {\it which monic polynomials with coefficients in $\mathbb Z$ are characteristic polynomials of simple graphs?}

Many necessary conditions are known. For instance, there exists a graph $G$ such that  $ \phi(G) = p(\lambda) =\lambda^n + \sum_{i=1}^n a_i\lambda^{i} \in {\mathbb Z}[\lambda]$ only if $a_1=0$, $-a_2$ is nonnegative and $-a_3$ is even and nonnegative (in fact, $a_1$, $-a_2$ and $-a_3$ respectively give the trace of $A(G)$, $\varepsilon_G$, and the numbers of the triangles in $G$ multiplied by $2$); moreover, all the roots of $p(\lambda)$ must be real (since $A(G)$ is symmetric).

Even taken together, all these conditions are far from being sufficient: the polynomial
\[ q(\lambda) = \lambda^6-6\lambda^2+7\lambda^2-2\]
is not the characteristic polynomial of a simple graph. Instead of running a program computing the characteristic
polynomial of all trangle-free simple graphs with order and size both equal to $6$, we can use Proposition \ref{main1}:
since $q(\lambda)\phi(P_5)=\phi(H_{11})$, (as a matter of fact, we used Propositions \ref{divi2} to determine $q(\lambda)$)  the existence of a graph $G$ such that $\phi(G)=q(\lambda)$ would be against the fact that $H_{11}$ is determined by its spectrum.

More generally, a polynomial $p(\lambda)$ cannot be  the characteristic polynomial of a simple graph if there exist
a pair of graphs $(G',G)$ such that $\phi(G')p(\lambda)=\phi(G)$ and $G$ is DS.

\section*{Acknowledgements}
The first two authors are supported by National Natural Science Foundation of China (No. 11971274).
The third and the fourth authors acknowledge the support of INDAM-GNSAGA.\\


\bigskip
\noindent {\bf \Large Appendices}
\appendix
\section{The lowest terms of some graphs}\label{SA}
From the equality $\phi(P_0)=1$ and Lemma \ref{path}, we easily obtain the following proposition.
\begin{prop}\label{veryzoz}
For each nonnegative integer $k$ and $\epsilon \in \{0,1,2,3\}$,
\begin{equation}\label{path-low} \scalemath{1}{\vartheta(P_{4k+\epsilon}) =
\begin{cases}
\phantom{-}1 & \text{if  $\epsilon =0$;} \\
\phantom{-}(2k+1)\lambda &  \text{if  $\epsilon =1$;} \\
-1, &  \text{if  $\epsilon =2$;}  \\
-2(k+1)\lambda  & \text{if  $\epsilon =3$.}
\end{cases}
}
\end{equation}
Moreover, for $k \in \mathbb N$ or $(k,\epsilon)=(0,3)$.
\begin{equation}\label{cycle-low} \scalemath{1}{
\vartheta(C_{4k+\epsilon}) =
\begin{cases}
-4k^2\lambda^2 & \text{if $\epsilon =0$};\\
-2 & \text{if $\epsilon\in \{1,3\}$};\\
-4  &  \text{if $\epsilon =2$}.
\end{cases}
}
\end{equation}
\end{prop}
We now give
the proofs of Propositions~\ref{lollipop-low}, \ref{T22c-low} and  \ref{DK-low}. The technique is essentially the same for all of them: we just use one of the Schwenk formul\ae \
given in Lemma~\ref{formula1} together with Proposition~\ref{veryzoz}. Recall that a vertex $v$ is said to be {\it quasi-pendant} if $N_G(v)$ contains a pendant vertex.\smallskip

\noindent{\bf Proof of Proposition \ref{lollipop-low}}. For $n=4s+2\epsilon_1$ and $g=4t-2\epsilon_2$, let $u$ be the only vertex of degree $3$ in $G=C_{n,g}^0$, and let $v$ be the only vertex in $N_G(u)\setminus C_g$.
Applying \eqref{schwenky2} with respect to the edge $e = uv$, we obtain
$
\phi(C_{n,g}^0) = \phi(C_g)\phi(P_n) - \phi(P_{g-1})\phi(P_{n-1}).
$
Therefore, $\vartheta(C_{n,g}^0)$ is given by the lowest term of the polynomial
\begin{equation}\label{zoz1} \vartheta(C_g)\vartheta(P_n) -
\vartheta(P_{g-1})\vartheta(P_{n-1}).
\end{equation}
Now, the statement follows once we plug into \eqref{zoz1} the values of $\vartheta(C_g)$, $\vartheta(P_n)$,  $\vartheta(P_{g-1})$ and $\vartheta(P_{n-1})$ as read in the several  rows of
Table~\ref{TAB1}, each of which has been filled in with the aid of~\eqref{path-low} and~\eqref{cycle-low}.\hfill{$\Box$}\smallskip

\begin{table}[h]
  \centering
\scalebox{0.9}{\begin{tabular}{cccccccccc}
      \toprule
& && \\[-.45cm]
$(\epsilon_1, \epsilon_2)$&&& $\vartheta(C_g)$&&$\vartheta(P_n)$ && $\vartheta(P_{g-1})$ &&$\vartheta(P_{n-1})$  \\[.1cm]
\midrule
&&&&\\[-.3cm]
 $(0,0)$    &&& $-4t^2\lambda^2$  && $\phantom{-}1$ && $-2t\lambda$ &&  $-2s\lambda$\\[.15cm]
 $(0,1)$    &&& $-4$  && $\phantom{-}1$ && $(2t+1)\lambda$ &&  $-2s\lambda$\\[.15cm]
 $(1,0)$     &&& $-4t^2\lambda^2$  && $-1$ && $-2t\lambda$ &&  $(2s-1)\lambda$\\[.15cm]
 $(1,1)$  &&& $-4$  && $-1$ && $(2t+1)\lambda$ &&  $(2s-1)\lambda$\\ [.15cm]
 \bottomrule
 \hline
\end{tabular} }
\caption{\small Values needed to compute  $\vartheta(C_{4s-\epsilon_1,4t+\epsilon_2}^0)$}\label{TAB1}
\end{table}

\noindent{\bf Proof of Proposition \ref{T22c-low}}. Let $u$ be the only vertex with degree $3$ in $T_n$, and let $v$ be a quasi-pendant vertex in $N_{T_n}(u)$.
Applying the Schwenk formula~\eqref{schwenky2} with respect to the edge $uv$, we obtain $\phi(T_n) = (\lambda^2-1) \left(\phi(P_{n-2}) - \lambda\phi(P_{n-5})\right)$. Thus, $\vartheta(T_n)$ is equal to the lowest term of $ \lambda\vartheta(P_{n-5})-\vartheta(P_{n-2})$. Now, by Proposition~\ref{veryzoz},
\[ \scalemath{1}{ \lambda\vartheta(P_{n-5}) -\vartheta(P_{n-2} ) =
\begin{cases}
-2(s-1)\lambda^2 +1 & \text{if  $n =4s$,} \\
\phantom{-}(2s+1)\lambda &  \text{if  $n =4s+1$,} \\
\;\;\;(2s-1)\lambda^2-1 &  \text{if  $n=4s+2$,}  \\
-2(s+1)\lambda  & \text{if  $n=4s+3$,}
\end{cases}
}
\]
from which \eqref{tetateta} immediately follows.  \hfill{$\Box$}

\medskip
\noindent{\bf Proof of Proposition \ref{DK-low}}.
For $n=10$, Equation~\eqref{DKDK} holds  by a direct computation.  We assume now $n=4s+\epsilon > 10$. Let $u$ be one of the two  vertices with degree $3$ in $H_n$ and let $v$ be a quasi-pendant vertex in $N_{H_n}(u)$.
 Applying the Schwenk formula~\eqref{schwenky2} with respect to  the edge $uv$, we have
\begin{equation*}
\phi(H_n) = (\lambda^2-1) \left( \phi \left(T_{n-2}\right)-\lambda \phi \left(T_{n-5}\right) \right)
\end{equation*}
(for the definition of the several $T_n$'s see \eqref{tienne}).
Therefore, $\vartheta(G)$ is the lowest term of the polynomial
\[ p(\lambda):= \lambda \vartheta \left(T_{n-5}\right) -\vartheta \left(T_{n-2}\right) =
\begin{cases}
-2(s-1)\lambda^2 +1 & \text{if  $n =4s$;} \\
\phantom{-}(2s+1)\lambda &  \text{if  $n =4s+1$;} \\
\;\;\;(2s-1)\lambda^2-1 &  \text{if  $n=4s+2$;}  \\
-2(s+2)\lambda  & \text{if  $n=4s+3$.}
\end{cases}
\]
The polynomial $p(\lambda)$ has been computed with the aid of Proposition~{\ref{T22c-low}}. \hfill{$\Box$}

\section{The tree $H_{11}=P_{2,2,7}^{2,4}$ is DS}\label{SB}
\begin{lem}\label{lem1zoz}
If a connected graph $G$ is cospectral to an H-shape graph, then $G$ is itself an H-shape graph.
\end{lem}
\begin{proof} Immediate from  \cite[Lemma 3.3]{liu_huang}.
\end{proof}
\begin{lem}\label{lem1}
	If $\phi(G) = \phi(H_{11})$,  then $G$ is connected.
\end{lem}
\begin{proof} By a direct computation, \[
		\phi(H_{11}) =
\lambda		(\lambda-1)^2(\lambda+1)^2(\lambda^2-3) (\lambda^4-5\lambda^2+2).
	\]
	Thus, the two largest eigenvalues in ${\rm sp} (H_{11})$ are $\rho_{11}= (\sqrt{10+2\sqrt{17}})/2$ and $\lambda_2(H_{11}) = \sqrt{3}$.

 Let  $G_1, G_2, \dots,  G_i $ be the connected components of $G$ and suppose by contradiction that $i>1$. It is not restrictive  to assume
$\rho(G_1) \geqslant \rho(G_2) \geqslant \cdots \geqslant \rho(G_i)$.
 By Proposition \ref{invariant} we know that $G$ is bipartite with
$\nu_G = \varepsilon_G+ 1 = 11$. Then, $G$ only contains cycles with even girth not exceeding $10$. By interlacing (i.e.\ Lemma~\ref{interlace}) $\sqrt{3}=\lambda_2(G) \geqslant \lambda_1(G_2)$. Therefore, only $G_1$ contains cycles (otherwise $\rho(G_2)$ would be larger than $2$).
We first show that $G_1$ is necessarily unicyclic.
If $G_1$ contains two cycles, then $i\geqslant 3$ and only $G_i$ is possibly an isolated point, since $\eta(G)=1$. Therefore, $\nu_{G_1}\leqslant 8$.
Furthermore, none of cycles of $G_1$ can be a quadrangle, otherwise $C_{1;4}^0 \subset G_1$, implying $\rho_{11}= \rho(G_1)> \rho(C_{1;4}^0)=\rho_{11}$, which is false.
This means that the bipartite non-unicyclic $G_1$ should contain two exagons sharing $3$ or $4$ edges; yet, in both cases, its spectral radius would be larger than $\rho_{11}$.

So far, we have proved that if $G$ is not connected, then $G=G_1 \cup G_2$, where
$G_1$ is unicyclic and $G_2$ is a tree.
 Moreover, always by interlacing, $G_2$ is either the star $K_{1,3}$ or a path $P_r$ with $r \leqslant 5$. It follows that
$G_2$ belongs to $\{P_1, P_2, P_5\}$, since $P_3,P_4$ and $K_{1,3}$ do not divide $H_{11}$. Consequently $\nu_{G_1} \in \{6, 9, 10\}$.
Let now $\mathfrak g$ be the girth of $G_1$. We already noted that  $\mathfrak g \in \{4,6,8,10\}$.
\smallskip

{\it Case 1.} ${\mathfrak g} = 4$. This case does not occur since $G_1$,
which has at least six vertices, would properly contain $C_{1;4}^0$; and we have already explained above why it is not possible. \smallskip

{\it Case 2.} ${\mathfrak g} = 6$. Surely $\nu_{G_1} \in \{9, 10\}$, since $G_1\not=C_6$, whose index is $2$. Since $\lambda_1(C_{1;6}^0)= \rho_{11}$, by  Lemma~\ref{proper}  $C_{1;6}^0$ cannot be a proper subgraph of $G_1$.  It follows that $G_1$ should contain  $C_{1,1,6}^{0,k}$ for a suitable $ k \in \{1,2,3\}$.
Yet, this is impossible, since $\min \left\{ \lambda_1(C_{1,1,6}^{0,k}) \; \big\vert \; 1 \leqslant k \leqslant 3 \right\} > 2,17 > \rho_{11}$.\smallskip

	{\it Case 3.} ${\mathfrak g} \in \{8, 10\}$. The unicyclic graphs with order $9$ or $10$ and girth in $\{8,10\}$ are in the set
$$ {\mathcal S}= \{ C_{1;8}^0,  C_{2;8}^0, C_{10} \} \cup \{ C_{1,1;8}^{0,k} \mid 1 \leqslant k \leqslant 4 \}.$$
The only graph in $\mathcal S$ whose index is $\tilde{\rho}$ is $C_{1,1;8}^{0,4}$; yet, $G_1$ could not be equal to $C_{1,1;8}^{0,4}$, since
$\sqrt{2}$ belongs to ${\rm sp} (C_{1,1;8}^{0,4}) \setminus {\rm sp}(G)$; hence,  $C_{1,1;8}^{0,4}$ does not divide $G$.
\end{proof}

We now have at our disposal all tools to prove Proposition \ref{main1}. Let $cw_6(G)$ denote the number of closed walks of length $6$ in a graph $G$, and let
 $\widetilde{G}$ be a graph cospectral to $H_{11}$. Clearly, $\widetilde{G}$ has $11$ vertices like $H_{11}$, which is a graph of type $\mathcal T_{10}$ (see Fig.~\ref{Fig3}).
Lemmas~\ref{lem1zoz} and~\ref{lem1} ensure that $\widetilde{G}$ is an H-shape graph. From Proposition~\ref{ahah}(ii) and Corollary~\ref{3match1}, we see that $G \in {\mathcal T}_{10} \cup {\mathcal T}_{11}$.
In \cite[pp. 6-7]{liu_huang} it is shown that  $cw_6(G') \not= cw_6(G'')$ for every pair $(G',G'') \in {\mathcal T}_{10} \times   {\mathcal T}_{11}$ such that $\nu_{G'}=\nu_{G''}$. Thus, by Proposition~\ref{invariant}(v)  we deduce $\widetilde{G} \in {\mathcal T}_{10}$. The proof ends by  observing  that the only graph with $11$ vertices in ${\mathcal T}_{10}$ is  $H_{11}$. \hfill{$\Box$}

\end{document}